\documentclass{scrartcl}

\KOMAoptions{
    fontsize = 11pt,
    DIV = 10,
    abstract = true,
}

\setkomafont{title}{\normalfont\bfseries}
\setkomafont{author}{\large}
\setkomafont{date}{\small}
\setkomafont{sectioning}{\bfseries}
\addtokomafont{section}{\large}
\addtokomafont{subsection}{\normalsize}

\usepackage[
    affil-it,
    ]{authblk}

\usepackage[
    backend=biber,
    bibstyle=numeric,
    citestyle=numeric-comp,
    sorting=nyt,
    maxbibnames=99,
    maxcitenames=2,
    giveninits=true,
    ]{biblatex}
\renewbibmacro{in:}{}
\usepackage[T1]{fontenc}
\usepackage[english]{babel}
\usepackage{csquotes}
\usepackage{microtype}

\usepackage{amsmath}
\usepackage{amssymb}
\usepackage{amsthm}
\usepackage{mathtools}
\usepackage{thmtools, thm-restate}

\usepackage[dvipsnames]{xcolor}

\usepackage{hyperref}
\usepackage[capitalise,noabbrev]{cleveref}
\hypersetup{%
    hidelinks=true,
	colorlinks=true,
	linkcolor={BrickRed},
    urlcolor={MidnightBlue},
    citecolor={ForestGreen},
}

\theoremstyle{plain}
\newtheorem{prelem}{Proposition}

\crefname{prelem}{Proposition}{Propositions}
\newtheorem{theorem}{Theorem}
\newtheorem{proposition}[theorem]{Proposition}
\newtheorem{lemma}[theorem]{Lemma}
\newtheorem{corollary}[theorem]{Corollary}

\theoremstyle{definition}

\newtheorem{problem}{Problem}
\crefname{problem}{Problem}{Problems}
\theoremstyle{remark}

\crefname{claim}{Claim}{Claims}

\usepackage{lineno}
\usepackage{todonotes}

\usepackage{enumerate}

\usepackage{tikz}
\usetikzlibrary{arrows,positioning,intersections,calc}
\usetikzlibrary{backgrounds}

\tikzstyle{hvertex}=[thick,circle,inner sep=0.cm, minimum size=1.5mm, fill=white, draw=black]
\tikzstyle{hedge}=[thick]
\colorlet{hellgrau}{black!30!white}
\colorlet{dunkelgrau}{black!70!white}

\usepackage{subcaption}
\usepackage{float}

\usepackage{multirow}
\usepackage{booktabs}
\usepackage{mleftright}

\newcommand{\lmult}{\ensuremath{\{\!\!\{}}
\newcommand{\rmult}{\ensuremath{\}\!\!\}}}

\usepackage{bbold}
\newcommand{\zero}{\ensuremath{\mathbb{0}}}
\newcommand{\one}{\ensuremath{\mathbb{1}}}
\newcommand{\onep}{\ensuremath{\mathbb{1}^\prime}}

\newcommand{\NP}{\ensuremath{\mathcal{NP}}}

\newenvironment{case}[2]{\smallskip\noindent\textbf{Case #1:} #2 \smallskip}

\begin{document}


\title{\Large Signed Total Roman Domination and Domatic Numbers: Degree Three and Complete Multipartite Graphs}

\author[1]{Arie M.C.A. Koster}
\author[2]{Lutz Volkmann}
\author[1]{Moritz Wehrmann}

\affil[1]{Teaching and Research Area Discrete Optimization,
        RWTH Aachen University,
        52056 Aachen, Germany.
        Email: {\normalfont\ttfamily koster@math2.rwth-aachen.de}, {\normalfont\ttfamily wehrmann@math2.rwth-aachen.de}}

\affil[2]{Institute for Geometry and Practical Mathematics,
        RWTH Aachen University,
        52056 Aachen, Germany.
        Email: {\normalfont\ttfamily volkm@math2.rwth-aachen.de}}

\date{\vspace*{-1em}\today}

\maketitle

\vspace*{-3em}
\begin{abstract}
    Signed total Roman domination is a variant of the classic Roman domination-problem in graphs.
    A \emph{signed total Roman dominating function (STRD function)} on a graph $G=(V,E)$ is a function $f: V \to \{-1,1,2\}$ such that (i) $\sum_{u \in N(v)} f(u) \geq 1$ for all $v \in V$, where $N(v)$ denotes the neighborhood of $v$, and (ii) every vertex $v$ with $f(v) = -1$ is adjacent to a vertex $u$ with $f(u) = 2$.
    The \emph{weight} of $f$ is $\sum_{v \in V} f(v)$.
    The \emph{signed total Roman domination number} of $G$ is the minimum weight among all its STRD functions.
    A \emph{signed total Roman dominating family (STRD family)} on $G$ is a family $\{f_1, \ldots, f_d\}$ of pairwise distinct STRD functions such that $\sum_{i=1}^{d} f_i(v) \leq 1$ for all $v \in V$.
    The \emph{signed total Roman domatic number} of $G$ is the maximum size among all its STRD families.

    In this paper, we relate the signed total Roman domination number of a cubic graph to its open packing number, $2$-tuple total domination number, and signed total domination number, allowing us to derive sharp bounds on the first invariant and to establish new $\mathcal{NP}$-completeness results for all four invariants.
    We demonstrate that having a degree-3 vertex determines a graph's signed total Roman domatic number.
    Combined with known results this implies that the associated decision problem is easy for graphs with maximum degree at most three and $\mathcal{NP}$-complete otherwise.
    To contrast these general hardness results, we determine signed total Roman domination and domatic numbers in complete multipartite graphs.
    Despite their simple structure, this is a non-trivial task.
\end{abstract}

\section{Introduction} \label{section: intro}

Let $G = (V(G),E(G))$ be a graph of \emph{order} $n(G) = |V(G)|$ and \emph{size} $m(G) = |E(G)|$.
For a vertex $v \in V(G)$ its \emph{neighborhood} is the set $N_{G}(v) = \{u : uv \in E(G)\}$ while its \emph{degree} is $d_{G}(G) = |N_{G}(v)|$; if the graph is clear from the context we simply write $N(v)$ and $d(v)$.
If $d(v) = r$ for all $v \in V(G)$ then $G$ is \emph{$r$-regular}; a $3$-regular graph is \emph{cubic}. 
The minimum and maximum degree among all vertices of $G$ are denoted $\delta(G)$ and $\Delta(G)$, respectively.
$G$ is \emph{isolate-free} if $\delta(G) \geq 1$.

A \emph{signed total Roman dominating function (STRD function)} on $G$ is a function $f: V(G) \to \{-1,1,2\}$ such that (i) $\sum_{u \in N(v)} f(u) \geq 1$ for each $v \in V(G)$ and (ii) each vertex $v$ with $f(v) = -1$ is adjacent to a vertex $u$ with $f(u) = 2$.
We write $V_{i} = \{v \in V : f(v) = i\}$ for $i \in \{-1,1,2\}$ and identify $f$ with the triple $(V_{-1}, V_{1}, V_{2})$.
An STRD function $f$ has \emph{weight} $\omega(f) = \sum_{v \in V} f(v) = 2|V_{2}| + |V_{1}| - |V_{-1}|$.
The \emph{signed total Roman domination number} $\gamma_{stR}(G)$ is the minimum weight among all STRD functions on $G$ and an STRD function with weight $\gamma_{stR}(G)$ is a \emph{$\gamma_{stR}(G)$-function}.

A \emph{signed total Roman dominating family (STRD family)} on $G$ is a family $\{f_1, \ldots, f_d\}$ of pairwise distinct STRD functions on $G$ such that $\sum_{i=1}^{d} f_i(v) \leq 1$ for all $v \in V$.
The \emph{signed total Roman domatic number} $d_{stR}(G)$ is the maximum size among all STRD families on $G$.
If $\delta(G) \geq 1$, then the function $f$ with $f(v) = 1$ for all $v \in V(G)$ is an STRD function and $\{f\}$ is an STRD family on $G$;
in this case $\gamma_{stR}(G)$ and $d_{stR}(G)$ are well-defined and $d_{stR}(G) \geq 1$. 

The signed total Roman domination and domatic number of graphs were introduced by \textcite{Volkmann16a,Volkmann16b}.
They are members of an extensively studied family of variations of the classic Roman domination problem \cite{CockayneDHH04}.
Other examples can be found in \cite{AbdollahzadehAhangarCS19,AbdollahzadehAhangarHSY16,SheikholeslamiV10,BeelerHH16,ChellaliHHM16,AbdollahzadehAhangarHLZS14} and for more information we refer to the recent surveys \cite{ChellaliJSV21,ChellaliJSV20,ChellaliJSV22,ChellaliJSV25}.

We continue the study of signed total Roman domination in graphs.
For cubic graphs $G$ we relate $\gamma_{stR}(G)$ to the open packing number $\rho^{o}(G)$, the $2$-tuple total domination number $\gamma_{2\times, t}(G)$, and the signed total domination number $\gamma_{st}(G)$ (see the next section for their definitions); in particular we establish the Gallai-type result $\gamma_{stR}(G) + \rho^{o}(G) = n(G)$.
Using these new relations we derive sharp bounds on $\gamma_{stR}(G)$ for cubic graphs $G$ and prove \NP-completeness of the decision problems associated with $\gamma_{stR}(G)$, $\rho^{o}(G)$, $\gamma_{2\times, t}(G)$, and $\gamma_{st}(G)$ for cubic planar and cubic bipartite graphs.
While all four problems are known to be hard for various graphs classes (including, for example, bipartite graphs and chordal graphs \cite{KosariRSAK21,HenningS99,daFonsecaRamosdSS14,Pradhan12,Henning04}), our complexity results for cubic graphs appear to be new.
For general graphs we show that the presence of a degree-$3$ vertex implies signed total Roman domatic number $1$ and characterize the complexity of deciding $d_{stR}(G)$ in terms of maximum degree: The problem is easy if $\Delta(G) \leq 3$ and \NP-complete if $\Delta(G) \geq 4$.
Given the general hardness of computing the signed total Roman domination or domatic number, it is of great interest to find graph classes where these invariants can be determined efficiently.
Complete graphs and complete bipartite graphs were treated in \cite{Volkmann16a,Volkmann16b,GuoVW25,ZhaoM17} and in this work, we study the class of complete $r$-partite graphs with $r \geq 3$.
For every graph $G$ in this class we determine $\gamma_{stR}(G)$, showing in particular that $\gamma_{stR}(G) \in \{2,3\}$; in the cases where $\gamma_{stR}(G) = 3$ we determine $d_{stR}(G)$ as well.

\section{Terminology and Helpful Results}

Generally, we follow the ``domination books'' \cite{HaynesHH20,HaynesHH21,HaynesHH23} for notation and terminology.
In the following let $G$ be a graph and $S \subseteq V(G)$.

$S$ is an \emph{open packing} of $G$ if $|N(v) \cap S| \leq 1$ for all $v \in V(G)$, that is, no vertex has two neighbors in $S$.
The maximum size among all open packings of $G$ is its \emph{open packing number $\rho^{o}(G)$}.
$S$ is a \emph{$2$-tuple total dominating set} of $G$ if $|N(v) \cap S| \geq 2$ for all $v \in V(G)$, that is, each vertex has two neighbors in $S$.
The minimum size among all $2$-tuple total dominating sets of $G$ is its \emph{$2$-tuple total domination number $\gamma_{2\times, t}(G)$}; it is well-defined if $\delta(G) \geq 2$.
For a function $f: V(G) \to \mathbb{R}$ we let $f(S) = \sum_{v \in S} f(v)$ and denote the \emph{weight} of $f$ with $\omega(f) = f(V(G))$.
A function $f : V(G) \to \{-1,1\}$ is a \emph{signed total dominating function} on $G$ if $f(N(v)) \geq 1$ for all $v \in V(G)$.
The minimum weight $\omega(f)$ among all signed total dominating functions $f$ on $G$ is its \emph{signed total domination number $\gamma_{st}(G)$}; it is well-defined if $\delta(G) \geq 1$.

The \emph{induced subgraph} $G[S]$ has vertex set $S$ and contains all edges from $G$ that have both ends in $S$. 
We write $P_{n}$ and $C_{n}$ for the path and cycle on $n$ vertices, respectively.
With $K_n$ we denote the complete graph of order $n$ and with $K_{p_1, \ldots, p_r}$ we denote the complete $r$-partite graph with partite sets of sizes $p_1, \ldots, p_r$; we call $K_{p_1,p_2}$ a complete bipartite graph.

For a positive integer $k$ we use the abbreviations $[k] = \{1,2,\ldots,k\}$ and $[k]_0 = [k] \cup \{0\}$.
Multisets are written in double braces $\lmult \cdot \rmult$.
To indicate multiplicity in multisets or tuples we write $k \times e$ for $k$ copies of the element $e$ (in tuples these copies appear consecutively); 
for example, $\lmult 2 \times 3, 4, 3\rmult = \lmult 3, 3, 3, 4\rmult$ and $(2 \times 3, 4, 3) = (3,3,4,3)$.
We adopt this notation for complete multipartite graphs;
for example, $K_{n \times 1} = K_n$, while $K_{n \times 2}$ is the cocktail party graph, and $K_{p,q\times 1}$ is a complete split graph whose vertices can be partitioned into an independent set of size $p$ and a clique of size $q$ such that every member of the independent set is adjacent to every member of the clique.

For the transpose of a matrix $M$ we write $M^T$ and if all entries of $M$ are real numbers, the sum over the entries in the $i$th row (respectively, $j$th column) is denoted with $\vert M \vert_{i}^{R}$ (respectively, $\vert M \vert_{j}^{C}$).

We will frequently use the following fundamental results.

\begin{prelem}[\cite{Volkmann16b}]
    \label{prop: str domatic number leq mindegree}
    If $G$ is a graph with $\delta(G) \geq 1$, then $d_{stR}(G) \leq \delta(G)$.
    Moreover, if $d_{stR}(G) = \delta(G)$, then for each STRD family $\{f_1, \ldots, f_d\}$ with $d = d_{stR}(G)$ and each vertex $v$ of minimum degree, $\sum_{u \in N(v)} f_i(u) = 1$ for all $i \in [d]$ and $\sum_{i=1}^{d} f_i(u) = 1$ for all $u \in N(v)$.
\end{prelem}

\begin{prelem}[\cite{Volkmann16b}]
    \label{prop: str domination number times str domatic number leq order}
    If $G$ is a graph of order $n$ with $\delta(G) \geq 1$, then $\gamma_{stR}(G) \cdot d_{stR}(G) \leq n$.
    Moreover, if $\gamma_{stR} \cdot d_{stR}(G) = n$, then for each STRD family $\{f_1, \ldots, f_d\}$ with $d = d_{stR}(G)$, each $f_i$ is a $\gamma_{stR}(G)$-function and $\sum_{i=1}^{d} f_i(v) = 1$ for all $v \in V(G)$.
\end{prelem}

Concerning the signed total Roman domination and domatic numbers of complete graphs and complete bipartite graphs the following is known: 

\begin{prelem}[\cite{Volkmann16a}]
    \label{prop: str domination number complete graph}
    If $n$ is an integer with $n \geq 3$, then $\gamma_{stR}(K_n) = 3$.
\end{prelem}

\begin{prelem}[\cite{ZhaoM17}]
    \label{prop: str domination number complete bipartite graph}
    If $p$ and $q$ are integers with $p \geq q \geq 1$, then $\gamma_{stR}(K_{p,q}) = 2$, except for the case $(p,q) = (3,3)$, where $\gamma_{stR}(K_{p,q}) = 4$, and the case $(p,q) \in \{(s, 1) : s \geq 2\} \cup \{(s,3) : s \geq 4\} \cup \{(3,2)\}$, where $\gamma_{stR}(K_{p,q}) = 3$.
\end{prelem}

\begin{prelem}[\cite{GuoVW25}]
    \label{prop: str domatic number complete graph}
    If $n$ is an integer with $n \geq 3$, then $d_{stR}(K_n) = \lfloor n/3 \rfloor$, except for the case $n \in \{7,9,10,11\}$, where $d_{stR}(K_n) = \lfloor n/3 \rfloor - 1$.
\end{prelem}

\begin{prelem}[\cite{GuoVW25}]
    \label{prop: str domatic number complete bipartite graph}
    If $p$ and $q$ are integers with $p \geq q \geq 1$, then $d_{stR}(K_{p,q})=q$, except for the case $p=3$ or $q=3$, where $d_{stR}(K_{p,q})=1$, the case $(p,q)=(5,4)$, where $d_{stR}(K_{p,q})=2$, and the case $(p,q)=(7,6)$, where $d_{stR}(K_{p,q})=3$.
\end{prelem}

\section{Signed Total Roman Domination and Degree-3 Vertices} \label{section: Signed Total Roman Domination and Vertex Degrees}

In this section we study the signed total Roman domination number in cubic graphs and the signed total Roman domatic number in graphs with at least one degree-3 vertex.

\begin{theorem}
    \label{thm: gallai type STR domination open packing}
    If $G$ is a cubic graph of order $n$, then $\gamma_{stR}(G) + \rho^{o}(G) = n$.
\end{theorem}

\begin{proof}
    First, let $S$ be an open packing with $|S| = \rho^{o}(G)$ and construct an STRD function $f = (V_{-1}, V_{1}, V_{2})$ as follows.
    Starting with $V_{2} = \emptyset$, for every $x \in S$, add one neighbor of $x$ to $V_{2}$.
    As $S$ is an open packing, no vertex is added to $V_{2}$ twice; hence, this process yields a set of size $|V_{2}| = |S|$.
    Afterward set $V_{-1} = S$ and $V_{1} = V(G) - (V_{2} \cup V_{-1})$.
    Since $V_{-1}$ is an open packing, every $v \in V(G)$ has at most one neighbor in $V_{-1}$ and at least two neighbors in $V_{1} \cup V_{2}$; hence $f(N(v)) \geq 1$.
    Further we ensured that every vertex from $V_{-1}$ has a neighbor in $V_{2}$.
    Therefore, $f$ is an STRD function on $G$ with weight $\omega(f) = 2|V_{2}| + |V_{1}| - |V_{-1}| = n - |S|$.
    Since $\gamma_{stR}(G) \leq \omega(f)$ and $|S| = \rho^{o}(G)$ we get $\gamma_{stR}(G) + \rho^{o}(G) \leq n$. 

    Conversely, let $f = (V_{-1}, V_{1}, V_{2})$ be a $\gamma_{stR}(G)$-function.
    Since $G$ is cubic, for every $v \in V(G)$ the condition $f(N(v)) \geq 1$ yields that $v$ has at most one neighbor in $V_{-1}$; in other words, $V_{-1}$ is an open packing.
    From this fact and the condition that every vertex from $V_{-1}$ has a neighbor in $V_{2}$ we derive $|V_{-1}| \leq |V_{2}|$.
    Using these observations and $|V_{1}| = n - |V_{2}| - |V_{-1}|$ we conclude
    \begin{equation*}
        \omega(f) = 2|V_{2}| + |V_{1}| - |V_{-1}| = n + |V_{2}| - 2|V_{-1}| \geq n - |V_{-1}| \geq n - \rho^{o}(G).
    \end{equation*}
    Since $\omega(f) = \gamma_{stR}(G)$ we get the reverse inequality $\gamma_{stR}(G) + \rho^{o}(G) \geq n$.
\end{proof}

Let $G$ be a cubic graph of order $n$.
From the definition of a $2$-tuple total dominating set and an open packing follows $\gamma_{\times 2, t}(G) + \rho^{o}(G) = n$ and for the signed total domination number \textcite{Henning04} showed $\gamma_{st}(G) + 2\rho^{o}(G) = n$.
Consequently, 
\begin{equation}
    \label{eq: chain for cubic graphs}
    \gamma_{stR}(G) = \gamma_{\times 2, t}(G) = \frac{1}{2}(n + \gamma_{st}(G)) = n - \rho^{o}(G).
\end{equation}
One aspect of \cref{eq: chain for cubic graphs} is that bounds in terms of order for $\gamma_{\times 2, t}(G)$, $\gamma_{st}(G)$, and $\rho^{o}(G)$ can be translated into bounds for $\gamma_{stR}(G)$.
In \cite{Henning04} the sharp bounds $n/3 \leq \gamma_{st}(G) \leq 5n/7$ were presented, while \textcite{HenningY10} gave the improved sharp upper bound $\gamma_{\times 2, t}(G) \leq 5n/6$ for every connected cubic graph different from the Heawood graph depicted in \cref{fig:Heawood graph}.
Thus: 

\begin{figure}[]
    \centering
    \begin{tikzpicture}
        \foreach \i in {0,...,13}{
            \node[hvertex] (\i) at (\i*360/14:1.25) {};
        }
        \draw[hedge] (0) -- (1) -- (2) -- (3) -- (4) -- (5) -- (6) -- (7) -- (8) -- (9) -- (10) -- (11) -- (12) -- (13) -- (0);
        \draw[hedge] (0) -- (5);
        \draw[hedge] (2) -- (7);
        \draw[hedge] (4) -- (9);
        \draw[hedge] (6) -- (11);
        \draw[hedge] (8) -- (13);
        \draw[hedge] (10) -- (1);
        \draw[hedge] (12) -- (3);
    \end{tikzpicture}
    \caption{The Heawood graph.}
    \label{fig:Heawood graph}
\end{figure}
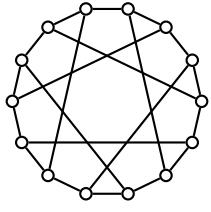

\begin{corollary}
   If $G$ is a cubic graph of order $n$, then $2n/3 \leq \gamma_{stR}(G) \leq 6n/7$ and these bounds are sharp; if additionally $G$ is connected and different from the Heawood graph, then $\gamma_{stR}(G) \leq 5n/6$ and this bound is sharp.
\end{corollary}

Another aspect of \cref{eq: chain for cubic graphs} concerns the decision problems associated with $\gamma_{stR}(G)$, $\gamma_{\times 2, t}(G)$, $\gamma_{st}(G)$, and $\rho^{o}(G)$ for cubic graphs $G$: If one of the problems is \NP-complete, then all four problems are.
Formally, we define \textsc{Signed Total Roman Domination}, \textsc{$2$-Tuple Total Domination}, \textsc{Signed Total Domination}, and \textsc{Open Packing} as the problems where, given a graph $G$ and an integer $k$, one has to decide whether $\gamma_{stR}(G) \leq k$, $\gamma_{\times 2, t}(G) \leq k$, $\gamma_{st}(G) \leq k$, or $\rho^{o}(G) \geq k$ holds, respectively.

For the following reductions recall that a \emph{subdivision} of an edge $e=vw$ is performed by deleting $e$ from $G$ and then adding a new vertex $x_{e}$ as well as edges $vx_{e}$ and $wx_{e}$.
The \emph{subdivision graph $S(G)$} is obtained by subdividing every edge of $G$ once.
We note that every independent set of $G$ is an open packing of $S(G)$ while for every open packing $B$ of $S(G)$, $B \cap V(G)$ is and independent set in $G$.

\begin{figure}[]
    \centering
    \begin{subfigure}[t]{.6\textwidth}
        \centering
        \begin{tikzpicture}[scale=1]

            \begin{scope}[shift={(-2.75,0)}]
            \node[hvertex, label=below:$y_{2}$] (2) at (0,0.25) {};
            \node[hvertex, label=above:$z_{1}$] (1') at (0,1) {};
            \node[hvertex, label=below:$z_{2}$] (2') at (-1,1.25) {};
            \node[hvertex, label=above:$z_{3}$] (3') at (-1,2.75) {};
            \node[hvertex, label=above:$z_{4}$] (4') at (1,2.75) {};
            \node[hvertex, label=below:$z_{5}$] (5') at (1,1.25) {};
            \node[hvertex, label=above:$z_{6}$] (6') at (-.3,1.9) {};
            \node[hvertex, label=above:$z_{7}$] (7') at (.3,1.9) {};

            \draw[hedge] (2) -- (1') -- (2') -- (3') -- (4') -- (5') -- (1');
            \draw[hedge] (2') -- (6') -- (3');
            \draw[hedge] (4') -- (7') -- (5');
            \draw[hedge] (6') -- (7');
            \end{scope}

            \begin{scope}[shift={(0,0)}]
            \node[hvertex, label=below:$y_{3}$] (3) at (0,0.25) {};
            \node[hvertex, label=above:$z_{1}'$] (1') at (0,1) {};
            \node[hvertex, label=below:$z_{2}'$] (2') at (-1,1.25) {};
            \node[hvertex, label=above:$z_{3}'$] (3') at (-1,2.75) {};
            \node[hvertex, label=above:$z_{4}'$] (4') at (1,2.75) {};
            \node[hvertex, label=below:$z_{5}'$] (5') at (1,1.25) {};
            \node[hvertex, label=above:$z_{6}'$] (6') at (-.3,1.9) {};
            \node[hvertex, label=above:$z_{7}'$] (7') at (.3,1.9) {};

            \draw[hedge] (3) -- (1') -- (2') -- (3') -- (4') -- (5') -- (1');
            \draw[hedge] (2') -- (6') -- (3');
            \draw[hedge] (4') -- (7') -- (5');
            \draw[hedge] (6') -- (7');
            \end{scope}

            \begin{scope}[shift={(2.75,0)}]
            \node[hvertex, label=below:$y_{4}$] (4) at (0,0.25) {};
            \node[hvertex, label=above:$z_{1}''$] (1') at (0,1) {};
            \node[hvertex, label=below:$z_{2}''$] (2') at (-1,1.25) {};
            \node[hvertex, label=above:$z_{3}''$] (3') at (-1,2.75) {};
            \node[hvertex, label=above:$z_{4}''$] (4') at (1,2.75) {};
            \node[hvertex, label=below:$z_{5}''$] (5') at (1,1.25) {};
            \node[hvertex, label=above:$z_{6}''$] (6') at (-.3,1.9) {};
            \node[hvertex, label=above:$z_{7}''$] (7') at (.3,1.9) {};

            \draw[hedge] (4) -- (1') -- (2') -- (3') -- (4') -- (5') -- (1');
            \draw[hedge] (2') -- (6') -- (3');
            \draw[hedge] (4') -- (7') -- (5');
            \draw[hedge] (6') -- (7');
            \end{scope}

            \node[hvertex, label=below:$y_{1}$] (1) at (0,-.75) {};
            \node[hvertex, label=below:$x_{e}$] (0) at (-1.5,-.75) {};

            \draw[hedge] (0) -- (1) -- (2) -- (3) -- (4) -- (1);
        \end{tikzpicture}
        \caption{The planar gadget.}
        \label{fig:planar gadget}
    \end{subfigure}
    \begin{subfigure}[t]{.38\textwidth}
        \centering
        \begin{tikzpicture}[scale=.8]
            \node[hvertex, label=above:$x_{e}$] (0) at (0,0) {};
            \node[hvertex, label=above:$y_{1}$] (1) at (1,0) {};
            \node[hvertex, label=above:$y_{2}$] (2) at (2,0) {};
            \node[hvertex, label=above:$y_{3}$] (3) at (3,0) {};
            \node[hvertex, label=above:$y_{4}$] (4) at (4,0) {};
            \node[hvertex, label=above:$y_{5}$] (5) at (5,0) {};
            \node[hvertex, label=above:$y_{6}$] (6) at (6,0) {};
            \node[hvertex, label=below:$x_{e'}$] (13) at (0,-1) {};
            \node[hvertex, label=below:$y_{12}$] (12) at (1,-1) {};
            \node[hvertex, label=below:$y_{11}$] (11) at (2,-1) {};
            \node[hvertex, label=below:$y_{10}$] (10) at (3,-1) {};
            \node[hvertex, label=below:$y_{9}$] (9) at (4,-1) {};
            \node[hvertex, label=below:$y_{8}$] (8) at (5,-1) {};
            \node[hvertex, label=below:$y_{7}$] (7) at (6,-1) {};

            \draw[hedge] (0) -- (1) -- (2) -- (3) -- (4) -- (5) -- (6) -- (7) -- (8) -- (9) -- (10) -- (11) -- (12) -- (13);
            \draw[hedge] (1) -- (12);
            \draw[hedge] (2) -- (11);
            \draw[hedge] (3) -- (10);
            \draw[hedge] (4) -- (7);
            \draw[hedge] (5) -- (8);
            \draw[hedge] (6) -- (9);
        \end{tikzpicture}
        \caption{The bipartite gadget.}
        \label{fig:bipartite gadget}
    \end{subfigure}
    \caption{}
\end{figure}

\begin{theorem}
    \label{thm:NPcompleteCubicPlanar}
    \textsc{Open Packing} is \NP-complete, even when restricted to cubic planar graphs.
\end{theorem}

\begin{proof}
    \textsc{Open Packing} clearly is in \NP, so it remains to show \NP-hardness.
    We devise a reduction from the \NP-complete \textsc{Independent Set}-problem for cubic planar graphs \cite[pp. 194-195]{GareyJ79}.
    For this, we construct a graph $H$ from the subdivision graph $S(G)$ of a given cubic planar graph $G$ by attaching the gadget shown in \cref{fig:planar gadget} to the vertex $x_{e}$, for every $e \in E(G)$.
    $H$ is cubic and planar as well.
    We claim that $G$ has an independent set of cardinality at least $k$ if and only if $H$ has an open packing of cardinality at least $k+8m$, where $m=m(G)$ is the size of $G$.

    First, suppose that $A$ is an independent set of $G$ with $|A|\geq k$.
    Let the set $B$ be obtained from $A$ by adding the eight vertices $y_{3}$, $y_{4}$, $z_{3}$, $z_{4}$, $z_{3}'$, $z_{4}'$ $z_{3}''$, $z_{4}''$ from every gadget used in the construction of $H$.
    This set $B$ is an open packing of $H$ with $|B| = |A| + 8m \geq k + 8m$.
    
    Conversely, suppose that $B$ is an open packing of $H$ with $|B| \geq k+8m$.
    For every $e \in E(G)$, the vertices of the gadget attached to $x_{e}$ can be partitioned into the eight sets $\{x_{e}, y_{2}, y_{4}\}$, $\{y_{1}, y_{3}\}$, $\{z_{1}, z_{3}, z_{6}\}$, $\{z_{2}, z_{4}, z_{5}, z_{7}\}$, $\{z_{1}', z_{3}', z_{6}'\}$, $\{z_{2}', z_{4}', z_{5}', z_{7}'\}$, $\{z_{1}'', z_{3}'', z_{6}''\}$, and $\{z_{2}'', z_{4}'', z_{5}'', z_{7}''\}$.
    For each set, any two members have a common neighbor in $H$, so $B$ contains at most one member from each set.
    Summing up over all $e \in E(G)$, these eight sets account for $\leq 8m$ members of $B$, while the remaining $\geq k$ members must lie in $V(G)$.
    Thus, $|V(G) \cap B| \geq k$, and we are finished since $V(G) \cap B$ is an independent set of $G$.
\end{proof}

\begin{theorem}
    \label{thm:NPcompleteCubicBipartite}
    \textsc{Open Packing} is \NP-complete, even when restricted to cubic bipartite graphs.
\end{theorem}

\begin{proof}
    Let $G$ be a cubic graph and let $G'$ be a copy of $G$; in particular, we denote the copy of edge $e$ with $e'$.
    The graph $H$ is obtained from $S(G) \cup S(G')$ by attaching the gadget shown in \cref{fig:bipartite gadget} to the vertices $x_{e}$ from $S(G)$ and $x_{e'}$ from $S(G')$, for every $e \in E(G)$.
    $H$ is cubic and bipartite; one partite set contains all vertices from $G$, $x_{e'}$ for every $e \in E(G)$, and the vertices $y_{i}$ with $i$ odd from every gadget, while the other partite set contains all vertices from $G'$, $x_{e}$ for every $e \in E(G)$, and the vertices $y_{i}$ with $i$ even from every gadget.
    We claim that $G$ has an independent set of cardinality at least $k$ if and only if $H$ has an open packing of cardinality at least $2k+4m$, where $m$ is the size of $G$.

    First, suppose that $A$ is an independent set of $G$ with $|A|\geq k$ and denote the corresponding subset of $V(G')$ with $A'$.
    Let the set $B$ be obtained from $A \cup A'$ by adding the four vertices $y_{2}$, $y_{5}$, $y_{8}$, $y_{11}$ from every gadget used in the construction of $H$.
    This set $B$ is an open packing of $H$ with $|B| = |A| + |A'| + 4m \geq 2k + 4m$.

    Conversely, suppose that $B$ is an open packing of $H$ with $|B| \geq 2k+4m$.
    For every $e \in E(G)$, the vertices of the gadget attached to $x_{e}$ can be partitioned into the four sets $\{x_{e}, y_{2}, y_{12}\}$, $\{x_{e'}, y_{1}, y_{11}\}$, $\{y_{3}, y_{5}, y_{7}, y_{9}\}$, and $\{y_{4}, y_{6}, y_{8}, y_{10}\}$.
    For each set, any two members have a common neighbor in $H$, so $B$ contains at most one member from each set.
    Summing up over all $e \in E(G)$, these four sets account for $\leq 4m$ members of $B$, while the remaining $\geq 2k$ members must lie in $V(G) \cup V(G')$.
    Thus, $\max\{|V(G) \cap B|, |V(G') \cap B|\} \geq k$, and we are finished since $V(G) \cap B$ and $V(G') \cap B$ both correspond to an independent set of $G$.
\end{proof}

\begin{corollary}
    \textsc{Signed Total Roman Domination}, \textsc{$2$-Tuple Total Domination}, and \textsc{Signed Total Domination} are  \NP-complete, even when restricted to cubic planar or cubic bipartite graphs.
\end{corollary}

In \cite[Examples 5 and 6]{Volkmann16a} the signed total Roman domination number of all paths and cycles was determined. 
Therefore, we can characterize the complexity of deciding $\gamma_{stR}(G)$ for isolate-free graphs $G$ in terms of $\Delta(G)$ as follows: 

\begin{corollary}
    \label{cor:complexity STR domination and max degree}
    $\gamma_{stR}(G)$ can be computed efficiently for isolate-free graphs $G$ with $\Delta(G) \leq 2$, while deciding $\gamma_{stR}(G)$ is \NP-complete for graphs with $\Delta(G) \geq 3$.
\end{corollary}

We remark that the same dichotomy holds for $\gamma_{2\times,t}(G)$, $\gamma_{st}(G)$, and $\rho^{o}(G)$.
Now we turn our attention to the effect a degree-3 vertex has on the signed total Roman domatic number of a graph.

\begin{theorem}
    \label{thm: d str vertex of degree 3}
    If a graph $G$ with $\delta(G) \geq 1$ contains a vertex of degree $3$, then $d_{stR}(G) = 1$.
\end{theorem}

\begin{proof}
    Let $v \in V(G)$ with $d(v) = 3$ and $N(v) = \{u_1, u_2, u_3\}$.
    \cref{prop: str domatic number leq mindegree} yields $d_{stR}(G) \leq \delta(G) \leq 3$, so we need to show $d_{stR}(G) \neq 2$ and $d_{stR}(G) \neq 3$.

    To see that $d_{stR}(G) \neq 2$, suppose, for the sake of contradiction, that $\{f_1, f_2\}$ is an STRD family on $G$.
    For each $j \in [3]$, from $f_1(u_j) + f_2(u_j) \leq 1$ we conclude $f_1(u_j) = -1$ or $f_2(u_j) = -1$.
    Thus, without loss of generality, we can assume $f_1(u_j) = -1$ for two $j \in [3]$, leading to the contradiction $f_1(N(v)) \leq 0$. 

    To see that $d_{stR}(G) \neq 3$, suppose, for the sake of contradiction, that $\{f_1, f_2, f_3\}$ is an STRD family on $G$.
    Thus, $d_{stR}(G) = \delta(G) = 3$, so \cref{prop: str domatic number leq mindegree} yields $\sum_{j=1}^{3}f_i(u_j) = 1$ and thus, $\lmult f_i(u_j) \mid j \in [3] \rmult = \lmult -1,1,1\rmult$ for each $i \in [3]$.
    In particular, $f_i(u) \neq 2$ for all $u \in N(v)$, hence $f_i(v) \geq 1$ for all $i \in [3]$, leading to the contradiction $\sum_{i=1}^{3} f_i(v) \geq 3$.
\end{proof}

This improves \cite[Corollary 6]{Volkmann16b}, which asserts $d_{stR}(G) \leq 2$ for every graph $G$ with $\delta(G) = 3$, and generalizes \cite[Corollary 12]{Volkmann16b}, which asserts $d_{stR}(G) = 1$ for every cubic graph $G$.

Similar to \cref{cor:complexity STR domination and max degree}, we can characterize the complexity of deciding $d_{stR}(G)$ for isolate-free graphs $G$ in terms of $\Delta(G)$: 
\begin{corollary}
    $d_{stR}(G)$ can be computed efficiently for isolate-free graphs $G$ with $\Delta(G) \leq 3$, while deciding $d_{stR}(G)$ is \NP-complete for graphs with $\Delta(G) \geq 4$.
\end{corollary}

\begin{proof}
Firstly, if $\delta(G) = 1$ or $\Delta(G) = 3$, then \cref{prop: str domatic number leq mindegree} and \cref{thm: d str vertex of degree 3} yield $d_{stR}(G) = 1$, respectively.
Secondly, if $\delta(G) = \Delta(G) = 2$, every component of $G$ is a cycle. For the cycle $C_{n}$ of order $n \geq 3$, \cite[Example 7 and Corollary 11]{Volkmann16b} state that $d_{stR}(C_{n}) = 2$ if $n \equiv 0 \pmod 4$ and $d_{stR}(C_{n}) = 1$ otherwise.
Therefore, if $G$ is $2$-regular, $d_{stR}(G) = 2$ if every component is a cycle whose length is a multiple of $4$ and $d_{stR}(G) = 1$ otherwise.
Finally, in \cite{WehrmannK26} it was shown that deciding $d_{stR}(G)$ is \NP-complete, even when restricted to $4$-regular graphs $G$.
\end{proof}

\section{Signed Total Roman Domination Numbers of Complete Multipartite Graphs} \label{section: Signed Total Roman Domination Number of Complete Multipartite Graphs}

In this section we determine the signed total Roman domination number of every complete multipartite graph with at least three partite sets.
We begin by stating this main result. 

\begin{theorem}
    \label{thm: str domination number complete multipartite graphs}
    If $G = K_{p_1, \ldots, p_r}$ with $r \geq 3$ and $p_1 \geq \cdots \geq p_r \geq 1$, then $\gamma_{stR}(G)=2$, except when $G$ is one of the graphs
    \begin{itemize}
         \item $K_{p,q \times 1}$ with $p \geq 1$ and $q \geq 2$, or
        \item $K_{2,2,q \times 1}$, $K_{3,3,q \times 1}$, $K_{4,2,q \times 1}$, or $K_{4,4,q \times 1}$ with $q$ odd and $q \geq 1$ , or
        \item $K_{3,2,q \times 1}$ or $K_{4,3,q \times 1}$ with $q$ even and $q \geq 2$, or
        \item $K_{3,3,3}$,
    \end{itemize}
    in which case $\gamma_{stR}(G) = 3$.
\end{theorem}

The proof requires some preparation.

\begin{lemma}
    \label{lemma: lower bound str domination number complete multpartite graph}
    If $G$ a complete $r$-partite graph with $r \geq 2$, then $\gamma_{stR}(G) \geq 2$.
    Moreover, if $\gamma_{stR}(G) = 2$ and $P_1, \ldots, P_r$ denote the partite sets of $G$, then for any $\gamma_{stR}(G)$-function $f$ holds $f(P_i) \leq 1$ for all $i \in [r]$ and $f(P_i) = 1$ for at least two $i \in [r]$.
\end{lemma}

\begin{proof}
    Let $f$ be a $\gamma_{stR}(G)$-function and let $v_i \in P_i$ for all $i \in [r]$.
    Using that $N(v_i) = (P_1 \cup \cdots \cup P_r)- P_i$ for all $i \in [r]$ we obtain
    \begin{equation*}
        r 
        \leq \sum_{i=1}^{r} f(N(v_i))
        = (r-1)\sum_{i=1}^{r} f(P_i)
        = (r-1) \omega(f).
    \end{equation*}
    Thus, $\omega(f) \geq r/(r-1) > 1$ and consequently $\gamma_{stR}(G) \geq 2$.
    Moreover, if $\gamma_{stR}(G) = 2$, then $f(P_i) = \omega(f) - f(N(v_i)) \leq 2-1$ for all $i \in [r]$.
    Combining this with $\sum_{i=1}^{r} f(P_i) = \omega(f) = 2$ we conclude $f(P_i) = 1$ for at least two $i \in [r]$.
\end{proof}

In view of \cref{lemma: lower bound str domination number complete multpartite graph} we can treat the graphs $G$ in \cref{thm: str domination number complete multipartite graphs} as follows:
Either we construct an STRD function on $G$ with weight $2$, thus proving $\gamma_{stR}(G)= 2$, or we show that weight $2$ cannot be achieved and then construct an STRD function on $G$ with weight $3$, thus proving $\gamma_{stR}(G)= 3$.
We introduce some notation to facilitate this.

Suppose that we want to construct an STRD function $f$ on a complete $r$-partite graph $G$ with partite sets $P_i$ of size $p_i$ for $i \in [r]$, and we want each $f(P_i) = \sum_{v \in P_i} f(v)$ to attain some target value $w_i \in \{-1,0,1,2\}$.
First, fix $i$ and let $(n_{-1}, n_1, n_2)$ denote the triple in row $p_i$ and column $w_i$ of \cref{table: weight assignment partition sets}.
In slight abuse of notation we write ``$f(P_i) \coloneqq w_i$'' to denote an operation where the values of $f$ on $P_i$ are set such that $n_{j}$ elements of $P_i$ are assigned value $j$, for $j \in \{-1,1,2\}$; after performing this operation, $f(P_i) = w_i$.
Second, we write ``$f(P_1, \ldots, P_r) \coloneqq (w_1, \ldots, w_r)$'' to simultaneously set $f(P_i) \coloneqq w_i$ for all $i \in [r]$; after performing this operation, $\omega(f) = \sum_{i=1}^{r} w_i$.

\begin{table}[h]
    \small
	\centering
	\begin{tabular}{c l | cccc }
		\toprule
         	&&\multicolumn{4}{c}{$w$}\\
		 && $-1$ & $0$ & $1$ & $2$ \\
		\midrule
		\multirow{ 7}{*}{$p$} & $1$ & $(1,0,0)$ & - & $(0,1,0)$ & $(0,0,1)$ \\
        		&$2$  & - & $(1,1,0)$ & $(1,0,1)$ & $(0,2,0)$ \\
        		&$3$  & $(2,1,0)$ & $(2,0,1)$ & $(1,2,0)$& $(1,1,1)$ \\
        		&$4$ & $(3,0,1)$ & $(2,2,0)$ & $(2,1,1)$ & $(2,0,2)$ \\
        		&$5$ & $(3,2,0)$ & $(3,1,1)$ & $(3,0,2)$ & $(2,2,1)$ \\
        		&$6+2t$& $(t+4,t+1,1)$ & $(t+4,t,2)$ & $(t+3,t+2,1)$ & $(t+3,t+1,2)$ \\
        		&$7+2t$& $(t+5,t,2)$ & $(t+4,t+2,1)$ & $(t+4,t+1,2)$ & $(t+4,t,3)$ \\
		\bottomrule
	\end{tabular}
	\caption{ The triple $(n_{-1}, n_{1}, n_{2})$ in row $p$ and column $w$ fulfills $n_{-1} + n_{1} + n_{2} = p$ and $-n_{-1} + n_{1} + 2n_{2} = w$.
    }
    \label{table: weight assignment partition sets}
\end{table}

Now we are prepared to tackle the proof of \cref{thm: str domination number complete multipartite graphs}.

\begin{proof}[Proof of \cref{thm: str domination number complete multipartite graphs}]
    Let $G=K_{p_1, \ldots, p_r}$ with $r \geq 3$ and $p_1 \geq \cdots \geq p_r$, and let $P_1, \ldots, P_r$ denote the partite sets of $G$ with $\vert P_i \vert = p_i$ for all $i \in [r]$.
    The proof breaks down into a case distinction over the number of partite sets of size at least $2$ and, if necessary, further distinction of certain values of $r$ and the $p_i$.

    \case{1}{$p_1 \geq 1$ and $p_i = 1$ for $2 \leq i \leq r$.}

	Note that $G = K_{p,q \times 1}$ with $p = p_1$ and $q = r-1$.
    We claim $\gamma_{stR}(G) = 3$.
    If $p_1 = 1$, then $G = K_r$ and $\gamma_{stR}(G) = 3$ by Proposition \ref{prop: str domination number complete graph}.
    From now on we assume $p_1 \geq 2$.
    Suppose, for the sake of contradiction, that $f$ is an STRD function on $G$ with $\omega(f) = 2$.
    By \cref{lemma: lower bound str domination number complete multpartite graph}, $f(P_i) \leq 1$ for all $i \in [r]$.
    Thus, $f(v) = -1$ for some $v \in P_1$ and thus $f(u) = 2$ for some neighbor $u$ of $v$, where $u \in P_i$ for some $i \in \{2, \ldots, r\}$;
    but $\vert P_i \vert = 1$ and $f(P_i)\leq 1$ yields the contradiction $2 = f(u) = f(P_i) \leq 1$.
    Hence, $\gamma_{stR}(G) \geq 3$, and in the following we define an STRD function $f$ on $G$ with $\omega(f) = 3$.
    If $r$ is even, say $r = 2s$ for an integer $s \geq 2$, then we set 
    \begin{equation*}
        f(P_1, \ldots, P_r) \coloneqq (1,2,(s-1)\times 1, (s-1)\times(-1)),
    \end{equation*}
    and if $r$ is odd, say $r = 2s+1$ for an integer $s \geq 1$, then we set
    \begin{equation*}
        f(P_1, \ldots, P_r) \coloneqq (0,2, s \times 1, (s-1) \times (-1)).
    \end{equation*}
    It remains to verify the correctness of the function $f$, which we do in three steps: 
    First, $\omega(f) = 3$ clearly holds.
    Second, for all $i \in [r]$ and $v \in P_i$, we have $f(N(v)) = \omega(f) - f(P_i) \geq 1$.
    Finally, with the help of \cref{table: weight assignment partition sets} it is readily verified that, for all $i \in [r]$ and $v \in P_i$, if $f(v) = -1$ then $f(u) = 2$ for some $u \in N(v) = V(G)- P_i$.
    Therefore, $f$ is indeed an STRD function on $G$ with the desired weight.
    In all following cases we perform this verification process without explicit mention.

    \case{2}{$p_1, p_2 \geq 2$ and $p_i = 1$ for $3 \leq i \leq r$.}

    Note that $G = K_{p_1,p_2,q \times 1}$ with $q = r-2$.
    
    \case{2.1}{$(p_1, p_2) \in \{(3,2), (4,3)\}$ and $r = 2s$ for an integer $s \geq 2$.}

    We claim that $\gamma_{stR}(G) = 3$.
    Since $K_{p_1, p_2, q \times 1} \cong K_{p_2, p_1, q \times 1}$, without loss of generality, we can assume $p_1 \in \{2,4\}$ and $p_2 = 3$.
    Suppose, for the sake of contradiction, that $f$ is an STRD function on $G$ with $\omega(f) = 2$.
    By \cref{lemma: lower bound str domination number complete multpartite graph}, $f(P_i) \leq 1$ for all $i \in [r]$.
    Thus, $f(v_1) = f(v_2) = -1$ for some $v_1 \in P_1$ and $v_2 \in P_2$, while $f(P_i) \in \{-1,1\}$ for all $i \in \{3, \ldots, r\}$.
    From the former we deduce $f(u_2) = 2$ for some $u_2 \in N(v_1)$, so the latter yields $u_2 \in P_2$; analogously we deduce $f(u_1) = 2$ for some $u_1 \in  P_1$.
    Now, from $|P_2| = 3$, $f(P_2) \leq 1$, and $f(v_2) = -1$ and $f(u_2) = 2$ for $v_2,u_2 \in P_2$ we conclude $f(P_2) = 0$; analogously, from $|P_1| \in \{2,4\}$, $f(P_1) \leq 1$, and $f(v_1) = -1$ and $f(u_1) = 2$ for $v_1,u_1 \in P_1$ we conclude $f(P_1) \in \{-1,1\}$.
    It follows that $\sum_{i=3}^{r} f(P_i) = \omega(f) - f(P_1) - f(P_2) \in \{1,3\}$.
    But this contradicts the fact that $\sum_{i=3}^{r} f(P_i)$ must be even, since $f(P_i) \in \{-1, 1\}$ for $i \in \{3,\ldots, r\}$ and $r$ is even.
    Hence, $\gamma_{stR}(G) \geq 3$, and we obtain an STRD function $f$ on $G$ with $\omega(f) = 3$ by setting
    \begin{equation*}
        f(P_1, \ldots, P_r) \coloneqq (1,2,(s-1) \times 1, (s-1) \times (-1)).
    \end{equation*}

    \case{2.2}{$(p_1, p_2) \notin \{(3,2), (4,3)\}$ and $r = 2s$ for an integer $s \geq 2$.}

	We claim that $\gamma_{stR}(G) = 2$.
    To show this, we define an STRD function $f$ on $G$ with $\omega(f) = 2$.
    If $p_1 = 3$ or $p_2 = 3$, then $p_1,p_2 \notin \{2,4\}$ and we set
    \begin{equation*}
        f(P_1, \ldots, P_r) \coloneqq (0,0,s \times 1, (s-2) \times (-1)).
    \end{equation*}
    Otherwise, $p_1, p_2 \neq 3$ and we set
    \begin{equation*}
        f(P_1, \ldots, P_r) \coloneqq ((s+1) \times 1, (s-1) \times (-1)).
    \end{equation*}

    \case{2.3}{$(p_1,p_2) \in \{(2,2), (3,3), (4,2), (4,4)\}$ and $r = 2s+1$ for an integer $s \geq 1$.}
    
    We claim that $\gamma_{stR}(G) = 3$.
    Suppose, for the sake of contradiction, that $f$ is an STRD function on $G$ with $\omega(f) = 2$.
    Analogous to Case 2.1 one can show the following; we omit the details.
    First, one can show that $f(P_1), f(P_2) \leq 1$ and $f(P_i) \in \{-1,1\}$ for $i \in \{3,\ldots, r\}$.
    One concludes that $(f(P_1), f(P_2)) \in \{(1,1), (1,-1), (-1,1), (-1,-1), (0,0)\}$ and therefore $\sum_{i=3}^{r} f(P_i) = \omega(f) - f(P_1) - f(P_2) \in \{0,2,4\}$.
    But this contradicts the fact that $\sum_{i=3}^{r} f(P_i)$ must be odd, since $f(P_i) \in \{-1,1\}$ for $i \in \{3,\ldots, r\}$ and $r$ is odd.
    Hence, $\gamma_{stR}(G) \geq 3$, and in the following we define an STRD function $f$ on $G$ with $\omega(f) = 3$.
    If $(p_1,p_2) \in \{(2,2), (4,2), (4,4)\}$, then we set 
    \begin{equation*}
        f(P_1, \ldots, P_r) \coloneqq ((s+2) \times 1, (s-1) \times (-1)).
    \end{equation*}
    Otherwise, $(p_1,p_2) = (3,3)$ and we set 
    \begin{equation*}
        f(P_1, \ldots, P_r) \coloneqq (2,2,(s-1) \times 1, s \times (-1)).
    \end{equation*}

    \case{2.4}{$(p_1,p_2) \notin \{(2,2), (3,3), (4,2), (4,4)\}$ and $r = 2s+1$ for an integer $s \geq 1$.}
    
    We claim that $\gamma_{stR}(G) = 2$.
    Since $K_{p_1, p_2, q \times 1} \cong K_{p_2, p_1, q \times 1}$, without loss of generality, we can assume $p_1 \notin \{2,4\}$ and $p_2 \neq 3$.
    We define an STRD function $f$ on $G$ with $\omega(f) = 2$ by setting
    \begin{equation*}
        f(P_1, \ldots, P_r) \coloneqq (0,(s+1) \times 1, (s-1) \times (-1)).
    \end{equation*}    

    \case{3}{$p_1, p_2, p_3 \geq 2$ and $p_i = 1$ for $4 \leq i \leq r$.}

    \case{3.1}{$(p_1, p_2, p_3) = (3,3,3)$ and $r = 3$.}

    Note that $G = K_{3,3,3}$.
    We claim $\gamma_{stR}(G) = 3$.
	Suppose, for the sake of contradiction, that $f$ is an STRD function on $G$ with $\omega(f) = 2$.
    From \cref{lemma: lower bound str domination number complete multpartite graph} we get $\lmult f(P_1), f(P_2), f(P_3)\rmult = \lmult 1,1,0\rmult$, say, without loss of generality, $f(P_1) = f(P_2) = 1$ and $f(P_3) = 0$.
    Thus, $\lmult f(v) \mid v \in P_1\rmult = \lmult f(v) \mid v \in P_2\rmult = \lmult 1,1,-1\rmult$ and $\lmult f(v) \mid v \in P_3\rmult = \lmult 2,-1,-1\rmult$.
    Thus, we have the contradiction that $f(v) = -1$ for some $v \in P_3$, but $f(u) \neq 2$ for all $u \in N(v) = P_1 \cup P_2$.
    Hence, $\gamma_{stR}(G) \geq 3$, and we construct an STRD function $f$ on $G$ with $\omega(f) = 3$ by setting
    \begin{equation*}
        f(P_1, P_2, P_3) \coloneqq (0,1,2).
    \end{equation*}

    \case{3.2}{$(p_1, p_2, p_3) = (3,3,3)$ and $r \geq 4$.}

    We claim that $\gamma_{stR}(G) = 2$ and construct an STRD function $f$ on $G$ with $\omega(f) = 2$ as proof.
    If $r$ is even, say $r = 2s$ for an integer $s \geq 2$, we set
    \begin{equation*}
        f(P_1, \ldots, P_r) \coloneqq (2 \times 0,s \times 1, (s-2) \times (-1)).
    \end{equation*} 
    Otherwise, $r$ is odd, say $r = 2s + 1$ for an integer $s \geq 2$, and we set
    \begin{equation*}
        f(P_1, \ldots, P_r) \coloneqq (3 \times 0,s \times 1, (s-2) \times (-1)).
    \end{equation*} 
    
    \case{3.3}{$(p_1, p_2, p_3) \neq (3,3,3)$.}
    
    We claim that $\gamma_{stR}(G) = 2$ and construct an STRD function $f$ on $G$ with $\omega(f) = 2$ as proof.
    Since $K_{p_1, p_2, p_3, p_4, \ldots, p_r} \cong K_{p_{\pi(1)}, p_{\pi(2)}, p_{\pi(3)},p_4, \ldots, p_r}$ for any permutation $\pi$ of $\{1,2,3\}$, we can assume, without loss of generality, that either $p_1 = 3$ and $p_2 \neq 3$, or $p_1, p_2, p_3 \neq 3$.
    First, assume that $r$ is even, say $r = 2s$ for an integer $s \geq 2$.
    If $p_1 = 3$ and $p_2 \neq 3$ we set 
    \begin{equation*}
        f(P_1, \ldots, P_r) \coloneqq (0,1,0,(s-1) \times 1, (s-2) \times (-1)),
    \end{equation*} 
    and if $p_1, p_2, p_3 \neq 3$ we set
    \begin{equation*}
        f(P_1, \ldots, P_r) \coloneqq ((s+1) \times 1, (s-1) \times (-1)).
    \end{equation*} 
    Now assume that $r$ is odd, say $r = 2s + 1$ for an integer $s \geq 1$. 
    Still assuming that $p_1 = 3$ and $p_2 \neq 3$, or that $p_1, p_2, p_3 \neq 3$ holds, we set
    \begin{equation*}
        f(P_1, \ldots, P_r) \coloneqq (0,(s+1) \times 1, (s-1) \times (-1)).
    \end{equation*}    

    \case{4}{$p_1, p_2, p_3, p_4 \geq 2$ and $r \geq 4$.}
    
    We claim that $\gamma_{stR}(G) = 2$ and proceed with induction on $r$ to construct an STRD function $f$ on $G$ with $\omega(f) = 2$.
    For the inductive base we consider $r = 4$ and $r = 5$, and begin by letting $r = 4$.
    If $\lmult p_1, p_2, p_3, p_4 \rmult$ contains two elements $3$, say $p_1 = p_2 = 3$, then we set
    \begin{equation*}
        f(P_1, \ldots, P_4) \coloneqq (0,0,1,1).
    \end{equation*}
    Otherwise, $\lmult p_1, p_2, p_3, p_4 \rmult$ contains two elements different from $3$, say $p_1 ,p_2 \neq 3$, and we set
    \begin{equation*}
        f(P_1, \ldots, P_4) \coloneqq (1,1,0,0).
    \end{equation*}
    
    Next let $r = 5$.
    If $\lmult p_1, p_2, p_3 \rmult$ contains two elements different from $2$ and $4$, we set
    \begin{equation*}
        f(P_1, \ldots, P_5) \coloneqq (0,0,0,1,1).
    \end{equation*}
    Otherwise, $\lmult p_1, p_2, p_3 \rmult$ contains two elements equal to $2$ or $4$, say $p_1,p_2 \in \{2,4\}$.
    Either $p_5 = 1$ and we set  
    \begin{equation*}
        f(P_1, \ldots, P_5) \coloneqq (1,1,1,0,-1),
    \end{equation*}
    or $p_5 \neq 1$ and we set
    \begin{equation*}
        f(P_1, \ldots, P_5) \coloneqq (1,1,0,0,0).
    \end{equation*}
    Note that for all functions $f$ defined above, $f(u^\ast) = 2$ for some $u^\ast \in P_1 \cup P_2 \cup  \cdots \cup P_r$.

    Finally, let $r \geq 6$ and let $H  = G[ P_1 \cup P_2 \cup \cdots \cup P_{r-2}] = K_{p_1, \ldots, p_{r-2}}$.
    As inductive hypothesis we assume $g$ to be an STRD function on $H$ with $\omega(g) = 2$ and with $g(u^\ast) = 2$ for some $u^\ast \in V(H)$. 
    For the inductive step we construct an STRD function $f$ on $G$ with $\omega(f) = 2$  and with $f(u^\ast) = 2$ for some $u^\ast \in V(G)$. 
    First we set $f(v) = g(v)$ for all $v \in V(H) = P_1 \cup P_2 \cup \cdots \cup P_{r-2}$, so in particular $f(u^\ast) = g(u^\ast) = 2$ for some $u^\ast \in V(H) \subseteq V(G)$; we fix $u^\ast$ for later reference.
    Next we define the values of $f$ on $P_{r-1} \cup P_{r}$.
    If $p_{r-1} = p_{r} = 2$, we set $f(P_{r-1}, P_{r}) \coloneqq (0,0)$.
    Otherwise, $p_{r-1} \neq 2$ or $p_{r} \neq 2$, say $p_{r-1} \neq 2$, and we set $f(P_{r-1}, P_{r}) \coloneqq (-1,1)$.
    Thus, $\omega(f) = \omega(g) + f(P_{r-1} \cup P_{r}) = 2 + 0$ and it remains to show that $f$ is an STRD function on $G$.
    First, let $v \in P_1 \cup P_2 \cup \cdots \cup P_{r-2}$.
    We have $f(N_{G}(v)) = g(N_{H}(v)) + f(P_{r-1} \cup P_{r}) \geq 1 + 0$.
    Further, if $f(v) = -1$, then $g(v) = -1$, which implies $g(u) = 2$ and thus $f(u) = 2$ for some $u \in N_H(v) \subseteq N_G(v)$.
    Next, let $v \in P_{r-1}$. 
    We have $f(N_{G}(v)) = \omega(g) + f(P_{r}) \geq 2 - 1$.
    Further, $f(u^\ast) = 2$ and $u^\ast \in N_{G}(v)$.
    The case where $v \in P_{r}$ is treated analogously. 
    Therefore, $f$ is indeed an STRD function on $G$.\smallskip
    
    We have considered all possible cases, so the proof of the theorem is complete.
\end{proof}

\section{Signed Total Roman Domatic Numbers of Complete Multipartite Graphs}
\label{section: Signed Total Roman Domatic Number of some Complete Multipartite Graphs}

In this section, we determine $d_{stR}(G)$ when $G$ is a complete multipartite graph with at least three partite sets and $\gamma_{stR}(G) = 3$, that is, when $G$ is one of the exceptions in \cref{thm: str domination number complete multipartite graphs}.
Note that the special case $G = K_{n}$ is already covered by \cref{prop: str domatic number complete graph}.

We state the main results and then lay out a plan to prove them.

\begin{restatable}{theorem}{FirstTheoremDSTR}
    \label{thm: d str K 3 3 3}
    $d_{stR}(K_{3,3,3}) = 1$
\end{restatable}

\begin{restatable}{theorem}{SecondTheoremDSTR}
    \label{thm: d str K 2 2 q times 1}
    If either $(p_{1},p_{2}) \in \{(2,2),(3,3),(4,2),(4,4)\}$ and $q$ is odd with $q \geq 1$, or $(p_{1},p_{2}) \in \{(3,2),(4,3)\}$ and $q$ is even with $q \geq 2$, then $d_{stR}(K_{p_{1},p_{2},q \times 1})= \lfloor (p_1+p_2+q)/3 \rfloor$, except for the case $p_1+p_2+q \in \{7,9,11\}$, where $d_{stR}(K_{p_{1},p_{2},q \times 1})= \lfloor (p_1+p_2+q)/3 \rfloor - 1$.
\end{restatable}

\begin{restatable}{theorem}{ThirdTheoremDSTR}
    \label{thm: d str K p q times 1}
    If $p$ and $q$ are integers with $p,q \geq 2$, then $d_{stR}(K_{p,q\times 1}) = \min\{q, \lfloor (p+q)/3 \rfloor\}$, except for the following cases:
    \begin{enumerate}[\normalfont (i)]
        \item If $q = 3$ or $p+q = 7$, then $d_{stR}(K_{p,q\times 1}) = 1$.
        \item If $q \neq 3$ and $p+q \in \{9,10,11\}$ or $(p,q) \in \{(7,5), (8,5), (9,5)\}$, then $d_{stR}(K_{p,q\times 1}) = 2$.
        \item If $(p,q) \in \{(11,7), (12,7), (13,7)\}$, then $d_{stR}(K_{p,q\times 1}) = 5$.
    \end{enumerate}
\end{restatable}

Since the graphs $G$ we consider in this section all have $\gamma_{stR}(G) = 3$, combining \cref{prop: str domatic number leq mindegree,prop: str domination number times str domatic number leq order} gives the upper bound
\begin{equation}
    \label{eq: d str leq minimum complete multipartite graph}
    d_{stR}(G) \leq \min \{\delta(G), \lfloor n(G)/3 \rfloor\}.
\end{equation}
Close inspection of \cref{thm: d str K 2 2 q times 1,thm: d str K p q times 1} shows that, with some exceptions, we claim that equality holds in \eqref{eq: d str leq minimum complete multipartite graph}.
Accordingly, their proofs break down into (a subset of) the following three cases.

In the first case, $d_{stR}(G) = \delta(G)$, and to prove this we only have to show $d_{stR}(G) \geq \delta(G)$, since \eqref{eq: d str leq minimum complete multipartite graph} yields the reverse inequality.
To that end we solve the problem of finding an STRD family $\{f_1, \ldots, f_d\}$ on $G$ of size $d = \delta(G)$ via an explicit construction of a matrix $M = (m_{iv})$ whose entry $m_{iv}$ encodes the function value $f_{i}(v)$ for all $i \in [d]$ and $v \in V(G)$.
In \cref{subsection: matrices} we list a variety of matrices that will act as building blocks for these constructions.

In the second case, $d_{stR}(G) = \lfloor n(G)/3 \rfloor$, and to prove this we only have to show $d_{stR}(G) \geq  \lfloor n(G)/3 \rfloor$, since \eqref{eq: d str leq minimum complete multipartite graph} yields the reverse inequality.
Here, too, we construct a matrix that encodes an STRD family on $G$ of size $\lfloor n(G)/3 \rfloor$, however, by adapting an inductive approach introduced by \textcite{GuoVW25}.
As inductive base this approach requires finding (matrices encoding) maximum size STRD families of numerous graphs of small order.
For this task we designed the binary integer linear program BIP1 presented in \cref{subsection:BIP}.

In the third case, $d_{stR}(G) = \beta$ for some $ \beta < \min \{\delta(G), \lfloor n(G)/3 \rfloor\}$. 
In \cref{subsection: auxiliary results} we present various results that will allow us to show $d_{stR}(G) \leq \beta$ while we can solve BIP1 to find an STRD family on $G$ of size $\beta$ to show the reverse inequality $d_{stR}(G) \geq \beta$.

The proofs of \cref{thm: d str K 3 3 3,thm: d str K 2 2 q times 1,thm: d str K p q times 1} will be presented in \cref{subsection:proofs}.

\subsection{Matrices}
\label{subsection: matrices}

The following matrices were introduced in \cite{GuoVW25}, unless explicitly stated otherwise.

First, \cref{tab: zero p times 2 etc} shows matrices $\zero_{(2s) \times 2}$ for integers $s \geq 1$, $\zero_{3 \times 3}$, and $\zero_{s \times 6}$ for integers $s \geq 2$; the latter are original to the present work.
The sum of each row and each column of these matrices is zero.

\begin{figure}[H]
    \centering
    \small
    \begin{subfigure}{0.45\textwidth}
        \centering
        $
        \mleft[
        \begin{array}{rr}
            1  &  -1    \\
            -1  &  1    \\
        \end{array}
        \mright]
        $
        \caption{$\zero_{2 \times 2}$}
    \end{subfigure}
    \begin{subfigure}{0.45\textwidth}
        \centering
        $
        \mleft[
        \begin{array}{rrr}
            2  &  -1  & -1  \\
            -1  &  2  & -1  \\
            -1  &  -1  & 2  \\
        \end{array}
        \mright]
        $
        \caption{$\zero_{3 \times 3}$}
    \end{subfigure}
    \newline
    \
    \newline
    \begin{subfigure}{0.33\textwidth}
        \centering
        \begin{tabular}{ c | cc}
            \toprule	
                &  $1$ & $2$\\
                \midrule
                $1$& \multicolumn{2}{c}{\multirow{2}{*}{$\zero_{2 \times 2}$}}\\
                $2$& &\\
                \multicolumn{1}{c|}{\multirow{2}{*}{\vdots}}& \multicolumn{2}{c}{\multirow{2}{*}{\vdots}}\\
                & & \\
                $2s-1$& \multicolumn{2}{c}{\multirow{2}{*}{$\zero_{2 \times 2}$}}\\
                $2s$& &\\
            \bottomrule
        \end{tabular}
        \caption{$\zero_{(2s) \times 2}$ for $s \geq 1$}
    \end{subfigure}
    \begin{subfigure}{0.3\textwidth}
        \centering
        \begin{tabular}{ c | cccccc}
            \toprule	
                &  $1$ & $2$ & $3$ & $4$ & $5$ & $6$\\
                \midrule
                $1$& \multicolumn{2}{c}{\multirow{2}{*}{$\zero_{2 \times 2}$}} & \multicolumn{2}{c}{\multirow{2}{*}{$\zero_{2 \times 2}$}} & \multicolumn{2}{c}{\multirow{2}{*}{$\zero_{2 \times 2}$}}\\
                $2$& & && &&\\
                \midrule
                $3$& \multicolumn{6}{c}{\multirow{4}{*}{$\zero_{(s-2) \times 6}$}} \\
                \multicolumn{1}{c|}{\vdots}& & && &&\\
                $s$& & && &&\\
            \bottomrule
        \end{tabular}
        \caption{$\zero_{s \times 6}$ for even $s \geq 2$}
    \end{subfigure}
     \begin{subfigure}{0.33\textwidth}
        \centering
        \begin{tabular}{ c | cccccc}
            \toprule	
                &  $1$ & $2$ & $3$ & $4$ & $5$ & $6$\\
                \midrule
                $1$& \multicolumn{3}{c}{\multirow{3}{*}{$\zero_{3 \times 3}$}} & \multicolumn{3}{c}{\multirow{3}{*}{$\zero_{3 \times 3}$}} \\
                $2$& & && &&\\
                $3$& & && &&\\
                \midrule
                $4$& \multicolumn{6}{c}{\multirow{4}{*}{$\zero_{(s-3) \times 6}$}} \\
                \multicolumn{1}{c|}{\vdots}& & && &&\\
                $s$& & && &&\\
            \bottomrule
        \end{tabular}
        \caption{$\zero_{s \times 6}$ for odd $s \geq 2$}
    \end{subfigure} 
    \caption{}
    \label{tab: zero p times 2 etc}
\end{figure}

Second, for integers $s \geq 1$ we define columns $u_{2s} = (2, (s-1) \times 1, s \times (-1))^T$ and $v_{2s+3} = (2,2, (s-1) \times 1, (s+2) \times (-1))^T$.
The matrix $\one_{(2s) \times (2s)}$ has $u_{2s}$ as first column, and for $j \in [2s-1]$, the $(j+1)$st column is obtained from the $j$th column by rotating down by one unit.
The matrix $\one_{(2s+3) \times (2s+3)}$ is obtained from $v_{2s+3}$ using the same method.
In \cref{tab: one s times s} we list $\one_{2 \times 2}$, $\one_{4 \times 4}$, and $\one_{5 \times 5}$ as examples.
The sum of each row and each column of these matrices is one, and each row contains an entry $2$.

\begin{figure}[H]
    \centering
    \small
    \begin{subfigure}{0.3\textwidth}
        \centering
        $
        \mleft[
        \begin{array}{rr}
            2  &  -1    \\
            -1  &  2    \\
        \end{array}
        \mright]
        $
        \caption{$\one_{2 \times 2}$}
    \end{subfigure}
    \begin{subfigure}{0.3\textwidth}
        \centering
        $
        \mleft[
        \begin{array}{rrrr}
            2  &  -1  & -1  &1 \\
            1  &  2  & -1 & -1 \\
            -1  &  1  & 2  & -1\\
            -1  &  -1  & 1 & 2 \\
        \end{array}
        \mright]
        $
        \caption{$\one_{4 \times 4}$}
    \end{subfigure}
        \begin{subfigure}{0.3\textwidth}
        \centering
        $
        \mleft[
        \begin{array}{rrrrr}
            2  &  -1  & -1  &-1 & 2\\
            2  &  2  & -1 & -1 &-1\\
            -1  &  2  & 2  & -1&-1\\
            -1  &  -1  & 2 & 2 &-1\\
            -1  &  -1  & -1 & 2 &2\\
        \end{array}
        \mright]
        $
        \caption{$\one_{5 \times 5}$}
    \end{subfigure}

    \caption{}
    \label{tab: one s times s}
\end{figure}

Third, \cref{tab: one 2s plus 1 times 2s plus 2} shows matrices $\onep_{5 \times 6}$ and $\onep_{(2s+1) \times (2s+2)}$ for integers $s \geq 3$; the latter are original to the present work.
The sum of each row and each column is one, except for the last column, whose sum is zero, and each row contains an entry $2$.

\begin{figure}[H]
    \centering
    \small
    \begin{subfigure}{0.4\textwidth}
        \centering
        $
        \mleft[
        \begin{array}{rrrrrr}
            2  &  -1  & -1  &1& -1 & 1\\
            -1  &  2  & -1 & 1& 1& -1\\
            2  &  -1  & 1  & -1& 1& -1\\
            -1  &  2  & 1  & -1& 1& -1\\
            -1  &  -1  & 1 & 1& -1& 2\\
        \end{array}
        \mright]
        $
        \caption{$\onep_{5 \times 6}$}
    \end{subfigure}
    \begin{subfigure}{0.58\textwidth}
    \centering
        \begin{tabular}{ c | cccccc | cc}
            \toprule	
                &  $1$ & \multicolumn{4}{c}{$\cdots$} & $2s$ & $2s+1$ & $2s+2$\\
                \midrule
                $1$& \multicolumn{6}{c|}{\multirow{3}{*}{$\onep_{(2s-1) \times (2s)}$}} & \multicolumn{2}{c}{\multirow{2}{*}{$\zero_{(2s-2) \times 2}$}} \\
                $\vdots$ &&&&&& &&\\
                $2s-1$ &&&&&& &$1$&$-1$\\
                \midrule
                $2s$& \multicolumn{4}{c}{\multirow{2}{*}{$\zero_{(2s-2) \times 2}^T$}} & \multicolumn{1}{c}{$1$} & $2$ & $-1$ & $-1$ \\
                $2s+1$& &&& & \multicolumn{1}{c}{$-1$} & $-1$ & $1$ & $2$ \\
            \bottomrule
        \end{tabular}
        \caption{$\onep_{(2s+1) \times (2s+2)}$ for $s \geq 3$}
        \label{tab: onep 2s plus 1}
    \end{subfigure}

    \caption{}
    \label{tab: one 2s plus 1 times 2s plus 2}
\end{figure}

Fourth, for integers $s \geq 5$, we define columns 
$z_{2s} = (-1,2,(s-1) \times 1, (s-1) \times (-1))^T$,
$e_{2s} = (3 \times 2,(s-5) \times 1, (s+1) \times (-1), 1)^T$, and
$o_{2s} = (3 \times 2,(s-4) \times 1, (s+1) \times (-1))^T$.
The matrix $\onep_{(2s) \times (2s+1)}$ has $z_{2s}$ as first column.
If $s$ is even (respectively, odd), then $\onep_{(2s) \times (2s+1)}$ has $e_{2s}$ (respectively, $o_{2s}$) as $(s+1)$st column.
The $(2s+1)$st column of $\onep_{(2s) \times (2s+1)}$ is $(1,-1,\ldots,1,-1)^T$ with alternating entries $1$ and $-1$.
The remaining columns are obtained as follows.
For $j \in [2s-1]- \{s\}$, the $(j+1)$st column is obtained from the $j$th column by rotating down by two units.
The sum of each row and each column is one, except for the last column, whose sum is zero, and each row contains an entry $2$.
In \cref{tab: one 2s times 2s plus 1} we list $\onep_{10 \times 11}$ and $\onep_{12 \times 13}$ as examples.

\begin{figure}[H]
    \centering
    \small
    \begin{subfigure}{0.455\textwidth}
        \centering
        $
        \mleft[
        \begin{array}{rrrrr rrrrr r}
            -1& -1& -1& 1&1&2& -1& -1& -1& 2& 1\\
            2& -1&-1& 1&1&2& -1& -1& -1& 1& -1\\
            1& -1& -1& -1& 1&2&2& -1& -1& -1& 1\\
            1& 2& -1& -1& 1&1&2& -1& -1& -1& -1\\
            1& 1& -1& -1& -1& -1& 2& 2& -1& -1& 1\\
            1&1&2& -1& -1& -1& 1 &2& -1& -1& -1\\
            -1& 1& 1& -1& -1& -1& -1& 2& 2& -1& 1\\
            -1 &1&1&2& -1& -1& -1& 1 &2& -1& -1\\
            -1& -1 &1 &1& -1& -1& -1& -1& 2&2&1\\
            -1& -1& 1&1&2& -1& -1& -1& 1& 2& -1\\
        \end{array}
        \mright]
        $
        \caption{$\onep_{10 \times 11}$}
    \end{subfigure}
    \begin{subfigure}{0.535\textwidth}
        \centering
        $
        \mleft[
        \begin{array}{rrrrr rrrrr rrr}
            -1& -1& -1& 1&1&1&2& -1& -1& -1& -1& 2& 1\\
            2& -1& -1& -1& 1&1&2&1& -1& -1& -1& 1& -1\\
            1& -1& -1& -1& 1&1&2&2& -1& -1& -1& -1& 1\\
            1& 2& -1& -1& -1& 1&1&2&1& -1& -1& -1& -1\\
            1& 1& -1& -1& -1& 1& -1& 2& 2& -1& -1& -1& 1\\
            1&1&2& -1& -1& -1& -1& 1&2&1& -1& -1& -1\\
            1&1&1& -1& -1& -1& -1& -1& 2& 2& -1& -1& 1\\
            -1& 1&1&2& -1& -1& -1& -1& 1&2&1& -1& -1\\
            -1& 1&1&1& -1& -1& -1& -1& -1& 2& 2& -1& 1\\
            -1& -1& 1&1&2& -1& -1& -1& -1& 1&2&1& -1\\
            -1& -1& 1&1&1& -1& -1& -1& -1& -1& 2&2&1\\
            -1& -1& -1& 1&1&2&1& -1& -1& -1& 1& 2& -1\\
        \end{array}
        \mright]
        $
        \caption{$\onep_{12 \times 13}$}
    \end{subfigure}
    \caption{}
    \label{tab: one 2s times 2s plus 1}
\end{figure}

Additionally, the three matrices $\onep_{4 \times 5}$, $\onep_{6 \times 7}$, and $\onep_{8 \times 9}$ are listed in \cref{tab: one 2s times 2s plus 1 contd}. 
The former two are original to the present work.
The sum of each row and each column is one, except for the last column, whose sum is zero. 
Only in $\onep_{8 \times 9}$ each row contains an entry $2$.

\begin{figure}[H]
    \centering
    \small
    \begin{subfigure}{0.25\textwidth}
        \centering
        $
        \mleft[
        \begin{array}{rrrrr}
            2  &  2  & -1  &-1& -1\\
            1  &  1  & -1 & -1& 1\\
            -1  &  -1  & 2  & 2& -1\\
            -1  &  -1  & 1 & 1& 1\\
        \end{array}
        \mright]
        $
        \caption{$\onep_{4 \times 5}$}
    \end{subfigure}
    \begin{subfigure}{0.3\textwidth}
        \centering
        $
        \mleft[
        \begin{array}{rrrrrrr}
            2  &  2  & 1  &-1 & -1 & -1 & -1\\
            1  &  1  & -1 & 1 &-1 &-1 &1\\
            -1  &  -1  & 2  & 2& 1 & -1 & -1\\
            -1  &  -1  & 1 & 1 &-1 & 1 &1\\
            1  &  -1  & -1 & -1 &2 &2 & -1\\
            -1  &  1  & -1 & -1 &1 & 1 &1\\
        \end{array}
        \mright]
        $
        \caption{$\onep_{6 \times 7}$}
    \end{subfigure}
    \begin{subfigure}{0.4\textwidth}
        \centering
        $
        \mleft[
        \begin{array}{rrrrrrrrr}
            2& 2& -1& -1& -1& 1& -1& 1& -1\\
            -1& -1& 2& 2& -1& 1& 1& -1& -1\\
            -1& -1& -1& 1& 1& -1& 2& -1& 2\\
            2& 2& -1& -1& 1& -1& -1& 1& -1\\
            -1& -1& 2& -1& 2& -1& -1& 1& 1\\
            -1& -1& -1& 1& 1& -1& -1& 2& 2\\
            2& 2& -1& -1& -1& 1& 1& -1& -1\\
            -1& -1& 2& 1& -1& 2& 1& -1& -1\\
        \end{array}
        \mright]
        $
        \caption{$\onep_{8 \times 9}$}
    \end{subfigure}

    \caption{}
    \label{tab: one 2s times 2s plus 1 contd}
\end{figure}

\subsection{BIP Model}\label{subsection:BIP}

We describe a binary integer linear program to compute $d_{stR}(G)$ for an isolate-free graph $G$.
The basic idea is to model an STRD family $\mathcal{F}$ on $G$ of maximum size as subset of a family $\{f_1, \ldots, f_{d}\}$ of appropriate functions $f_{i}$ with $d$ sufficiently large.
By \cref{prop: str domatic number leq mindegree,prop: str domination number times str domatic number leq order}, $d = \min\{\delta(G), \lfloor n(G) / \gamma_{stR}(G) \rfloor\}$ suffices (or $d = \delta(G)$ if $\gamma_{stR}(G)$ is unknown).

Let $I = \{1,2,\ldots, d\}$, $V = V(G)$, and $W = \{-1,1,2\}$.
We have three kinds of variables:
First, for each $i \in I$, $v \in V$, and $w \in W$ we have a binary variable $x_{i,v,w}$, where $x_{i,v,w} = 1$ indicates that $f_{i}(v) = w$;
second, for each $i \in I$ we have a binary variable $y_{i}$, where $y_{i} = 1$ indicates that $f_i \in \mathcal{F}$;
and third, for each pair $i,j \in I$ with $i < j$ and each $v \in V$ we have a binary variable $z_{i,j,v}$, where $z_{i,j,v} = 1$ indicates that $f_i(v) \neq f_j(v)$.
The program reads as follows.

\begin{align}
    \nonumber
    \textbf{BIP1} &&& &&\\
    \label{bip: obj}
    \max &&& \sum_{i \in I} y_{i} &&\\
    \label{bip: constr 2}
    \text{s.t.} &&& \sum_{w \in W} x_{i,v,w} = y_{i} &&\forall i \in I, v \in V\\
    \label{bip: constr 3}
    &&& \sum_{u \in N_{G}(v)} \sum_{w \in W} w \cdot x_{i,u,w}  \geq y_{i} && \forall i \in I, v \in V\\
    \label{bip: constr 4}
    &&& \sum_{u \in N_{G}(v)} x_{i,u,2}  \geq x_{i,v,-1} && \forall i \in I, v \in V\\
    \label{bip: constr 5}
    &&& \sum_{i \in I} \sum_{w \in W} w \cdot x_{i,v,w}  \leq 1 && \forall v \in V\\
    \label{bip: constr 6}
    &&& y_{i} + y_{j} - 1  \leq  \sum_{v \in V} z_{i,j,v}&& \forall i,j \in I, i < j\\
    \label{bip: constr 7}
    &&&  x_{i,v,w} + x_{j,v,w}  \leq 2 - z_{i,j,v} && \forall  i,j \in I, i < j, v \in V, w \in W\\
    \label{bip: constr 8}
    &&&  x_{i,v,w}, y_i \in \{0,1\} && \forall i \in I, v \in V, w \in W\\
    \label{bip: constr 10}
    &&&  z_{i,j,v}  \in \{0,1\} && \forall i,j \in I, i < j, v \in V
\end{align}

The objective \eqref{bip: obj} is to maximize the number of $i \in I$ with $y_i = 1$, that is, to maximize the size of $\mathcal{F}$.
When $y_i = 1$, that is, $f_i \in \mathcal{F}$, then constraints \eqref{bip: constr 2}, \eqref{bip: constr 3}, and \eqref{bip: constr 4} ensure that $f_i$ is an STRD function on $G$.
On the other hand, when $y_i = 0$, then \eqref{bip: constr 2} forces $x_{i,v,w} = 0$ for all $v \in V$ and $w \in W$, so in this case \eqref{bip: constr 3} and \eqref{bip: constr 4} are trivially fulfilled.
Next, constraint \eqref{bip: constr 5} ensures that $\sum_{f_i \in \mathcal{F}} f_i(v) \leq 1$ for all $v \in V$.
Thus, for $\mathcal{F}$ to be STRD family on $G$, it remains to ensure that the functions in $\mathcal{F}$ are pairwise distinct.
Let $i,j \in I$ with $i<j$.
If $y_i=y_j=1$, that is, $f_i, f_j \in \mathcal{F}$, then constraints \eqref{bip: constr 6} and \eqref{bip: constr 7} warrant $f_i(v) \neq f_j(v)$ for some $v \in V$.
Otherwise, setting $z_{i,j,v} = 0$ for all $v \in V$ trivially fulfills \eqref{bip: constr 6} and \eqref{bip: constr 7}.
We conclude that BIP1 correctly computes $d_{stR}(G)$.

\subsection{Auxiliary Results}
\label{subsection: auxiliary results}

The next, quite technical result shows that an entry $2$ in each row of $\onep_{4\times 5}$ or $\onep_{6\times 7}$ cannot be achieved (if we require that all entries are from $\{-1,1,2\}$ and that the sum of each row and column is at most one).
For its proof make frequent use of \cref{table: possible triples}, where, for $p \in [7]$ and $w \in \{-1,0,1,2\}$, all triples $(n_{-1}, n_{1}, n_{2})$ of non-negative integers with $n_{-1}+n_1+n_2=p$ and $-n_{-1}+n_1+2n_2=w$ are listed.
We illustrate this with an example:
Suppose that we are given a $(4 \times 5)$-matrix $M$ with entries from $\{-1,1,2\}$ and the entries of the first row of $M$ sum up to $1$.
From the entry for $(p,w) = (5,1)$ of \cref{table: possible triples} we conclude that the first row of $M$ either contains three entries $-1$ and two entries $2$, or it contains two entries $-1$ and three entries $1$. 

\begin{table}[]
    \small
	\centering
	\begin{tabular}{c l | cccc }
		\toprule
         &&\multicolumn{4}{c}{$w$}\\
		 && $-1$ & $0$ & $1$ & $2$ \\
		\midrule
		\multirow{ 7}{*}{$p$} & $1$ & $(1,0,0)$ & - & $(0,1,0)$ & $(0,0,1)$ \\
        &$2$  & - & $(1,1,0)$ & $(1,0,1)$ & $(0,2,0)$ \\
        &$3$  & $(2,1,0)$ & $(2,0,1)$ & $(1,2,0)$& $(1,1,1)$ \\
        &$4$ & $(3,0,1)$ & $(2,2,0)$ & $(2,1,1)$ & $(2,0,2),(1,3,0)$ \\
        &$5$ & $(3,2,0)$ & $(3,1,1)$ & $(3,0,2),(2,3,0)$ & $(2,2,1)$ \\
        &$6$ & $(4,1,1)$ & $(4,0,2),(3,3,0)$ & $(3,2,1)$ & $(3,1,2),(2,4,0)$ \\
        &$7$ & $(5,0,2),(4,3,0)$ & $(4,2,1)$ & $(4,1,2), (3,4,0)$ & $(4,0,3),(3,3,1)$ \\
		\bottomrule
	\end{tabular}
	\caption{All (!) triples $(n_{-1}, n_{1}, n_{2})$ of non-negative integers  with $n_{-1}+n_1+n_2=p$ and $-n_{-1}+n_1+2n_2=w$ for $p \in [7]$ and $w \in \{-1,0,1,2\}$.
    }
    \label{table: possible triples}
\end{table}

\begin{lemma}
    \label{lemma: 4 times 5 and 6 times 7 matrix one prime does not exist}
    For $k \in \{4,6\}$, a $(k \times (k+1))$-matrix $M$ with entries from $\{-1,1,2\}$ cannot simultaneously satisfy
    \begin{enumerate}[\normalfont (i)]
        \item $\vert M \vert^{R}_{i} \geq 1$ for all $i \in [k]$, 
        \item $\vert M \vert^{C}_{j} \leq 1$ for all $j \in [k+1]$, and
        \item each row contains an entry $2$.
    \end{enumerate}
\end{lemma}

\begin{proof}
    Suppose, for the sake of contradiction, that such a matrix $M$ satisfies (i), (ii), and (iii).
    From $\sum_{i=1}^{k} \vert M \vert^{R}_{i} = \sum_{j=1}^{k+1} \vert M \vert^{C}_{j}$, (i), and (ii) we derive
    \begin{enumerate}[\normalfont (a)]
        \item $\vert M \vert^{R}_{i} = 2$ for at most one $i \in [k]$ and $\vert M \vert^{R}_{i} = 1$ for all other $i \in [k]$, and
        \item $\vert M \vert^{C}_{j} = 0$ for at most one $j \in [k+1]$ and $\vert M \vert^{C}_{j} = 1$ for all other $j \in [k+1]$.
    \end{enumerate}

    First, assume $k = 4$.
    On the one hand, from (a), (iii), and the entries for $(p,w) = (5,1), (5,2)$ of \cref{table: possible triples} we conclude that at most one row of $M$ contains two entries $1$, and all other rows each contain no entry $1$.
    Thus, $M$ contains at most two entries $1$.
    On the other hand, from (b) and the entries for $(p,w) = (4,0), (4,1)$ of \cref{table: possible triples} we conclude that at most one column of $M$ contains two entries $1$, and all other columns each contain one entry $1$.
    Thus, $M$ contains at least five entries $1$, a contradiction.

    Next, assume $k = 6$.
    On the one hand, from (a), (iii), and the entries for $(p,w) = (7,1), (7,2)$ of \cref{table: possible triples} we conclude that at most one row of $M$ contains either one or three entries $2$, and all other rows each contain two entries $2$.
    Thus, $M$ contains at least eleven entries $2$.
    On the other hand, from (b) and the entries for $(p,w) = (6,0), (6,1)$ of \cref{table: possible triples} we conclude that at most one column of $M$ contains two entries $2$, and all other columns each contain one entry $2$.
    Thus, $M$ contains at most eight entries $2$, a contradiction.
\end{proof}

The two next results marginally improve the bound $d_{stR}(G) \leq \lfloor n/\gamma_{stR}(G)\rfloor$ from \cref{prop: str domination number times str domatic number leq order} for some graphs $G$.

\begin{proposition}
    \label{prop: str domatic number when str domination number geq n/3}
    If $G$ is a graph of order $n$ with $\delta(G) \geq 1$ and $\gamma_{stR}(G) \geq \lfloor n/3 \rfloor$, then $d_{stR}(G) \leq 2$.
\end{proposition}

\begin{proof}
    Let $G$ be a graph of order $n$ with $\delta(G) \geq 1$ and $\gamma_{stR}(G) \geq \lfloor n/3 \rfloor$.

    First assume $n \leq 5$.
    By \cref{prop: str domatic number leq mindegree}, $d_{stR}(G) \leq \delta(G)$, so the desired result follows immediately when $\delta(G) \leq 2$.
    Now consider graphs with $\delta(G) \geq 3$.
    These are precisely all $G \in \{K_4, K_5, K_{2,1,1,1}, K_{2,2,1} \}$;
    for these graphs we have $\gamma_{stR}(G) = 3$ according to \cref{thm: str domination number complete multipartite graphs}, and thus $d_{stR}(G) \leq n/\gamma_{stR}(G) < 2$ as required.

    From now on we can assume $n \geq 6$.
    This assumption yields $d_{stR}(G) \leq n / \gamma_{stR}(G) \leq  n / \lfloor n/3 \rfloor < 4$, so it remains to show that $d_{stR}(G) \neq 3$.
    Suppose, for the sake of contradiction, that $\{f_1, f_2, f_3\}$ is an STRD family on $G$.
    For each $v \in V(G)$, from $\sum_{i=1}^{3} f_i(v) \leq 1$ we derive for the multiset $\lmult f_i(v) \mid i \in [3] \rmult$ that (i) it contains at least one $-1$ and (ii) if it contains a $2$, then the other two elements are $-1$.
    By (i), $f_i(v) = -1$ for some $v \in V(G)$ and $i \in [3]$, say, without loss of generality, for $i = 1$.
    Thus, $f_1(v_1) = 2$ for some $v_1 \in V(G)$ and thus, by (ii), $f_2(v_1) = f_3(v_1) = -1$.
    Thus, $f_2(v_2) = f_3(v_3) = 2$ for some neighbors $v_2$ and $v_3$ of $v_{1}$ and thus, by (ii), $f_1(v_2) = f_3(v_2) = f_1(v_3) = f_2(v_3) = -1$.
    In particular, the function values show that $v_1$, $v_2$, and $v_3$ are pairwise distinct.
    Now, letting $S = \{v_1, v_2, v_3\}$, we have
    \begin{equation*}
        3\cdot \lfloor n/3 \rfloor
        \leq \sum_{i=1}^{3} \omega(f_i)
        = \sum_{v \in S} \sum_{i=1}^{3} f_i(v) + \sum_{v \in V(G) - S} \sum_{i=1}^{3} f_i(v) 
        \leq 0 + (n-3).
    \end{equation*}
    But this yields the contradiction $\lfloor n/3 \rfloor \leq n/3 -1$.
    Hence, $d_{stR}(G) \leq 2$ as required.
\end{proof}

\begin{proposition}
    \label{lemma: str domatic number when gamma str approx n/2}
    If $G$ is a graph of order $n$ with $\delta(G) \geq 1$ and $\gamma_{stR}(G) > 2 \lfloor n/2 \rfloor - \lceil n/2 \rceil$, then $d_{stR}(G) =1$.
\end{proposition}

\begin{proof}
    The only graphs of order $n \leq 3$ without isolated vertices are $G \in \{K_2,K_3,K_{2,1}\}$, where $d_{stR}(G) = 1$ by \cref{prop: str domatic number complete graph,prop: str domatic number complete bipartite graph}.
    From now on we assume $n \geq 4$.
    If $n$ is even, then $2 \lfloor n/2 \rfloor - \lceil n/2 \rceil = n/2$, so $d_{stR}(G) \leq n / \gamma_{stR}(G) < 2$ and thus, $d_{stR}(G) = 1$.
    From now on we assume that $n \geq 4$ is odd.
    Due to this assumption, $\gamma_{stR}(G) > 2 \lfloor n/2 \rfloor - \lceil n/2 \rceil$ yields $\gamma_{stR}(G) \geq (n-1)/2$.
    Thus, $d_{stR}(G) \leq n / \gamma_{stR}(G) \leq 2n/(n-1) < 3$,
    and thus, $d_{stR}(G) \leq 2$.
    Suppose, for the sake of contradiction, that $d_{stR}(G) = 2$ and let $\{f_1, f_2\}$ be an STRD family on $G$.
    For each $v \in V(G)$, from $f_1(v) + f_2(v) \leq 1$ we conclude $f_1(v) = -1$ or $f_2(v) = -1$.
    Thus, without loss of generality, we can assume $f_1(v) = -1$ for $\lceil n/2 \rceil$ vertices $v \in V(G)$.
    Since $f_1(v) \leq 2$ for the remaining $\lfloor n/2 \rfloor$ vertices $v \in V(G)$, we conclude $\gamma_{stR}(G) \leq \omega(f_1) \leq 2 \lfloor n/2 \rfloor - \lceil n/2 \rceil$, a contradiction.
    Hence, $d_{stR}(G) = 1$ as required.
\end{proof}

\subsection{Proofs of \cref{thm: d str K 3 3 3,thm: d str K 2 2 q times 1,thm: d str K p q times 1}}
\label{subsection:proofs}

For the reader's convenience we restate the theorems before giving their proofs.

\FirstTheoremDSTR*
\begin{proof}
    By \cref{prop: str domatic number when str domination number geq n/3}, $d_{stR}(K_{3,3,3}) \leq 2$.
    To the end of showing $d_{stR}(K_{3,3,3}) \neq 2$ suppose, for the sake of contradiction, that $\{f_1, f_2\}$ is an STRD family on $K_{3,3,3}$.
    For each $v \in V(K_{3,3,3})$, from $f_1(v) + f_2(v) \leq 1$ we conclude $f_1(v) = -1$ or $f_2(v) = -1$.
    Thus, without loss of generality, we can assume $f_1(v) = -1$ for (at least) five $v \in V(K_{3,3,3})$.
    These five vertices are divided among the three partite sets $P_1$, $P_2$, and $P_3$ of $K_{3,3,3}$, and thus, without loss of generality, we can assume $f_1(v) = -1$ for at least four $v \in P_1 \cup P_2$.
    But this leads to the contradiction that $f_1(N(v)) = f_1(P_1 \cup P_2) \leq 0$ for every $v \in P_3$.
\end{proof}

\SecondTheoremDSTR*
\begin{proof}
    We only treat $K_{p_1, p_2, q \times 1}$ with $(p_1, p_2) = (2,2)$ and odd $q \geq 1$; proofs for all other cases work analogously.
    Let $P_1 = \{v_1, v_2\}$ and $P_2 = \{v_3, v_4\}$ denote the two partite sets of size $2$ and let $Q = \{w_1, \ldots, w_q\}$ denote the union of the $q$ partite sets of size $1$.
    Note that $N(v) = P_2 \cup Q$ for $v \in P_1$, that $N(v) = P_1 \cup Q$ for $v \in P_2$, and that $N(w) = (P_1 \cup P_2 \cup Q)- \{w\}$ for $w \in Q$.
    
    \case{1}{$q \in \{1,3\}$.}
    
    By \cref{lemma: str domatic number when gamma str approx n/2}, $d_{stR}(K_{2,2,q \times 1}) = 1$ as required.
    
    \case{2}{$q \in \{5,7\}$.}

    By \cref{prop: str domatic number when str domination number geq n/3}, $d_{stR}(G) \leq 2$.
    It is readily verified that $\{f_1, f_2\}$, where $f_1$ (respectively, $f_2$) alternatingly assigns values $-1$ and $2$ (respectively, $2$ and $-1$) to the vertices $v_1, v_2, v_3, v_4, w_1, \ldots, w_q$, is an STRD family on $K_{2,2,q \times 1}$.
    Hence, $d_{stR}(K_{2,2,q \times 1}) = 2$ as required.

    \case{3}{$q$ is odd and $q \geq 9$.}

    \Cref{eq: d str leq minimum complete multipartite graph} reads as $d_{stR}(K_{2,2,q \times 1}) \leq \min\{q+2, \lfloor (q+4)/3 \rfloor\} = \lfloor (q+4)/3 \rfloor$.
    To demonstrate that this upper bound is attained, we inductively construct a matrix $\mathbb{M}(2,2,q)$ encoding an STRD family on $K_{2,2,q \times 1}$ of size $\lfloor (q+4)/3 \rfloor$.
    For $q \in \{9,11,13\}$ we obtained $\mathbb{M}(2,2,q)$ by solving BIP1; this establishes the inductive base.
    For the inductive step, \cref{tab: induction STRD family on K 2 2 q} shows how to construct $\mathbb{M}(2,2,q+6)$ from $\mathbb{M}(2,2,q)$.
    Hence, for all odd $q \geq 9$, $d_{stR}(K_{2,2,q \times 1}) = \lfloor (q+4)/3 \rfloor$ as required.
\end{proof}

\begin{figure}[]
        \centering
        \small
        \begin{tabular}{ c | cc  cc  cccccc | cccccc}
            \toprule	
                & $v_1$ &  \multicolumn{1}{c|}{$v_2$} & $v_3$ & \multicolumn{1}{c|}{$v_4$}  & $w_1$ & \multicolumn{4}{c}{$\cdots$} & $w_q$ & $w_{q+1}$ & \multicolumn{4}{c}{$\cdots$} & $w_{q+6}$\\
                \midrule
                $f_1$ & \multicolumn{10}{c|}{\multirow{3}{*}{$\mathbb{M}(2,2,q)$}} & \multicolumn{6}{c}{\multirow{3}{*}{$\zero_{d \times 6}$}}\\
                $\vdots$ &&&&& &&&&& &&&&& &\\
                $f_d$ &&&&& &&&&& &&&&& &\\
                \midrule
                $f_{d+1}$ & \multicolumn{7}{c}{\multirow{2}{*}{$\zero_{(q+1) \times 2}^T$}} & $1$ & $1$ & $1$ & $-1$ & $-1$ & \multicolumn{2}{c}{\multirow{2}{*}{$\one_{2 \times 2}$}} & \multicolumn{2}{c}{\multirow{2}{*}{$\one_{2 \times 2}$}} \\
                $f_{d+2}$ &&&&& &&&$-1$&$-1$&$-1$ &$2$&$2$&&& &\\
            \bottomrule
        \end{tabular}
        \caption{An STRD family $\{f_1, \ldots, f_{d+2}\}$ on $K_{2,2,(q+6) \times 1}$ with odd $q \geq 9$ and $d = \lfloor (q+4)/3 \rfloor$.
        The entries form the matrix $\mathbb{M}(2,2,q+6)$.}
        \label{tab: induction STRD family on K 2 2 q}
    \end{figure}

\ThirdTheoremDSTR*

\begin{proof}
    For $K_{p, q \times 1}$, let $P=\{v_1, \ldots, v_p\}$ denote the partite set of size $p$ and let $Q = \{w_1, \ldots, w_q\}$ denote the union of the $q$ partite sets of size $1$.
    Note that $N(v) = Q$ for all $v \in P$ and that $N(w) = (P \cup Q) - \{w\}$ for all $w \in Q$.
    The proof breaks down into a case distinction over $p$ and $q$.

    \case{1}{$p+q \leq 5$ or $p+q = 7$.}
    
    By \cref{lemma: str domatic number when gamma str approx n/2}, $d_{stR}(K_{p,q \times 1}) = 1$ as required.
    
    \case{2}{$q = 3$.}

    Since $d(v) = q = 3$ for every $v \in P$, \cref{thm: d str vertex of degree 3} yields $d_{stR}(G) = 1$ as required.

    \case{3}{$q \neq 3$ and $p+q \in \{9,10,11\}$.}

    We claim that $d_{stR}(K_{p,q \times 1}) = 2$.
    By \cref{prop: str domatic number when str domination number geq n/3}, $d_{stR}(K_{p,q \times 1}) \leq 2$.
    It is readily verified that $\{f_1, f_2\}$, where $f_1$ (respectively, $f_2$) alternatingly assigns values $-1$ and $2$ (respectively, $2$ and $-1$) to the vertices $v_1, \ldots, v_p, w_1, \ldots, w_q$, is an STRD family on $K_{p,q \times 1}$.
    Hence, $d_{stR}(K_{p,q \times 1}) = 2$ as required.

    \case{4}{$(p,q) \in \{(7,5), (8,5), (9,5)\}$.}

    We claim that $d_{stR}(K_{p,q \times 1}) = 2$.
    Since $d_{stR}(K_{p,q \times 1}) \leq \lfloor (p+q)/3 \rfloor\ = 4$, it suffices to show $d_{stR}(K_{p,q \times 1}) \neq 3,4$ and to construct an STRD family on $K_{p,q \times 1}$ of size $2$.

    To the end of showing $d_{stR}(K_{p,q \times 1}) \neq 3$ suppose, for the sake of contradiction, that $\{f_1, f_2, f_3\}$ is an STRD family on $K_{p,q \times 1}$.
    We distinguish two cases.

    Suppose first that for each $i \in [3]$, $f_i(w) = 2$ for some $w \in Q$, say $f_1(w_{k_1}) = f_2(w_{k_2}) = f_3(w_{k_3}) = 2$ for $k_1, k_2, k_3 \in [q]$.
    Since $\sum_{i=1}^{3} f_i(w_k) \leq 1$ for each $k \in \{k_1, k_2, k_3\}$, $f_1(w_{k_2}) =f_1(w_{k_3}) =f_2(w_{k_1}) =f_2(w_{k_3}) =f_3(w_{k_1}) =f_3(w_{k_2}) = -1$; in particular, $k_1$, $k_2$, and $k_3$ are pairwise distinct.
    Therefore, with $S = \{k_1, k_2, k_3\}$, we obtain the contradiction
    \begin{equation*}
        3 
        \leq \sum_{i=1}^{3} f_i(N(v_1))
        = \sum_{k \in S} \sum_{i=1}^{3} f_i(w_k) + \sum_{k \in [q] - S} \sum_{i=1}^{3} f_i(w_k)
        \leq 0 + (q-3) = 2.
    \end{equation*}

    Suppose next that for some $i \in [3]$, $f_i(w) \neq 2$ for all $w \in Q$, say, without loss of generality, for $i=1$.
    We show that $\omega(f_1) \geq 9$.
    If $f_1(v) \neq 2$ for all $v \in P$, then $f_1(x) \neq 2$ for all $x \in P \cup Q$, hence $f_1(x) = 1$ for all $x \in P \cup Q$, hence $\omega(f_1) =  p+q \geq 9$.
    Otherwise, $f_1(v) = 2$ for some $v \in P$, and from the assumption $f_1(w) \neq 2$ for all $w \in Q$ we derive $f_1(v) \neq -1$ for all $v \in P$, so $f_1(P) \geq p+1$.
    Since $f_1(Q) = f_1(N(v_1)) \geq 1$, we have $\omega(f_1) = f_1(P) + f_1(Q) \geq p+2 \geq 9$.
    Now, using $\omega(f_1) \geq 9$ we obtain the contradiction
        \begin{equation*}
        15
        = 9 + 2\gamma_{stR}(K_{p,q \times 1})
        \leq \sum_{i=1}^{3} \omega(f_i)
        = \sum_{x \in P \cup Q} \sum_{i=1}^{3} f_i(x)
        \leq p+q \leq 14.
    \end{equation*}
    
    Since all possibilities lead to a contradiction, $d_{stR}(K_{p,q \times 1}) \neq 3$.

    To the end of showing $d_{stR}(K_{p,q \times 1}) \neq 4$ suppose, for the sake of contradiction, that $\{f_1, f_2, f_3, f_4\}$ is an STRD family on $K_{p,q \times 1}$.
    Consider the $(4 \times 5)$-matrix $M=(f_i(w))_{i \in [4], w \in Q}$.
    Since $\vert M \vert_{i}^{R} = f_i(N(v_1)) \geq 1$ for all $i \in [4]$ and $\vert M \vert_{w}^{C} = \sum_{i=1}^{4} f_i(w) \leq 1$ for all $w \in Q$, \cref{lemma: 4 times 5 and 6 times 7 matrix one prime does not exist} yields that some row of $M$ contains no entry $2$, say, without loss of generality, $f_1(w) \neq 2$ for all $w \in Q$.
    Thus, $f_1(v) \neq -1$ for all $v \in P$, so $f_1(P) \geq p$.
    Since $f_1(Q) = f_1(N(v_1)) \geq 1$ we have $\omega(f_1) = f_1(P) + f_1 (Q) \geq p+1 \geq 8$.
    This leads to the contradiction
        \begin{equation*}
        17
        = 8 + 3\gamma_{stR}(K_{p,q \times 1})
        \leq \sum_{i=1}^{4} \omega(f_i)
        = \sum_{x \in P \cup Q} \sum_{i=1}^{4} f_i(x)
        \leq  p+q \leq 14.
    \end{equation*}
    Thus, $d_{stR}(K_{p,q \times 1}) \neq 4$.

    Finally, it is readily verified that $\{f_1, f_2\}$, where $f_1$ (respectively, $f_2$) alternatingly assigns values $-1$ and $2$ (respectively, $2$ and $-1$) to the vertices $v_1, \ldots, v_p, w_1, 
    \ldots, w_q$, is an STRD family on $K_{p,q \times 1}$.
    Hence, $d_{stR}(K_{p,q \times 1}) = 2$ as required.

    \case{5}{$(p,q) \in \{(11,7), (12,7), (13,7)\}$.}

    We claim that $d_{stR}(K_{p,q \times 1}) = 5$.
	Since $d_{stR}(K_{p,q \times 1}) \leq \lfloor (p+q)/3 \rfloor\ = 6$, it suffices to show $d_{stR}(K_{p,q \times 1}) \neq 6$ and to construct an STRD family on $K_{p,q \times 1}$ of size $5$.

    To the end of showing $d_{stR}(K_{p,q \times 1}) \neq 6$ suppose, for the sake of contradiction, that $\{f_1, \ldots, f_6\}$ is an STRD family on $K_{p,q \times 1}$.
    Proceeding like we did in Case 4, we consider the $(6 \times 7)$-matrix $M=(f_i(w))_{i \in [6], w \in Q}$ and use \cref{lemma: 4 times 5 and 6 times 7 matrix one prime does not exist} to conclude that some row of $M$ contains no entry $2$, say, without loss of generality, $f_1(w) \neq 2$ for all $w \in Q$.
    Then, similar to Case 4, we conclude $\omega(f_1) \geq p+1 \geq 12$.
    But this leads to the contradiction
    \begin{equation*}
        27
        = 12 + 5\gamma_{stR}(K_{p,q \times 1})
        \leq \sum_{i=1}^{6} \omega(f_i)
        = \sum_{x \in P \cup Q} \sum_{i=1}^{6} f_i(x) 
        \leq  p+q \leq 20.
    \end{equation*}
    Thus, $d_{stR}(K_{p,q \times 1}) \neq 6$.

    For each pair $(p,q) \in \{(11,7), (12,7), (13,7)\}$ we solved BIP1 for $K_{p,q \times 1}$ to obtain an STRD family on $K_{p,q \times 1}$ of size $5$.
    Hence, $d_{stR}(G) = 5$ as required.\\

    In the final two cases we cover the pairs $(p,q)$ where $d_{stR}(K_{p,q \times 1}) = \min \{q, \lfloor (p+q)/3 \rfloor\} \geq 2$.
    Note that $q < \lfloor (p+q)/3 \rfloor$ if and only if $p \geq 2q+3$.

    \case{6}{$p \geq 2q + 3$ and Cases 1 to 5 do not apply.}

    We claim that $d_{stR}(K_{p,q \times 1}) = q$.
    Among the pairs $(p,q)$ with $p \geq 2q + 3$ and $p,q \geq 2$ those excluded because one of Cases 1 to 5 applies are precisely the pairs with $q = 3$.
    For the remaining pairs $(p,q)$, \cref{tab: STRD family on K p q when p geq 2q} (on p.~\pageref{tab: STRD family on K p q when p geq 2q}) shows an STRD family on $K_{p,q \times 1}$ of size $q$.
    Hence, $d_{stR}(K_{p,q\times 1}) = q$ as required.

    \case{7}{$p \leq 2q + 2$ and Cases 1 to 5 do not apply.}
    
    We claim that $d_{stR}(K_{p,q \times 1}) = \lfloor (p+q)/3 \rfloor$.
    Among the pairs $(p,q)$ with $p \leq 2q + 2$ and $p,q \geq 2$, those excluded because one of Cases 1 to 5 applies are indicated by a $\heartsuit$ in \cref{tab: inductive base d str K p q} (on p.~\pageref{tab: inductive base d str K p q}).
	For the remaining pairs $(p,q)$ we inductively construct a matrix $\mathbb{M}(p,q) = [\mathbb{P}(p,q) \vert \mathbb{Q}(p,q)]$ encoding an STRD family on $K_{p,q\times 1}$ of size $\lfloor (p+q)/3 \rfloor$, where $\mathbb{P}(p,q)$ encodes the function values on $P$ and $\mathbb{Q}(p,q)$ encodes the function values on $Q$.
	For all pairs $(p,q)$ indicated by a $\clubsuit$ in \cref{tab: inductive base d str K p q} we obtained $\mathbb{M}(p,q)$ by solving BIP1 for $K_{p,q \times 1}$; this establishes the inductive base.
	As inductive step, \cref{tab: STRD family on K p q plus 6,tab: STRD family on K p plus 12 q plus 6} (on pp.~\pageref{tab: STRD family on K p plus 12 q plus 6} and \pageref{tab: STRD family on K p q plus 6}) show how to obtain $\mathbb{M}(p,q+6)$ and $\mathbb{M}(p+12,q+6)$ from $\mathbb{M}(p,q)$.
    Note that the following pairs $(p,q)$ with $p \leq 2q + 2$ and $p,q \geq 2$ are covered via the inductive step: Pairs indicated by a $\diamondsuit$ in \cref{tab: inductive base d str K p q} and pairs with $p \geq 26$ or $q \geq 18$.
    \smallskip

	We covered all possible cases, so the proof of the theorem is complete.
\end{proof}

    \begin{figure}[h]
    \centering
    \small
    \begin{subfigure}{\textwidth}
        \centering
        \begin{tabular}{ c | cccc | ccccc | ccc}
            \toprule
            & $v_1$ & \multicolumn{2}{c}{$\cdots$} &$v_{2q}$ & $v_{2q+1}$ & \multicolumn{3}{c}{$\cdots$} &$v_{p}$ &$w_1$ & $\cdots$ & $w_{q}$	\\	
            \midrule
            $f_1$& \multicolumn{2}{c}{\multirow{3}{*}{$\one_{q \times q}$}} & \multicolumn{2}{c|}{\multirow{3}{*}{$\one_{q \times q}$}} &  \multicolumn{2}{c}{\multirow{3}{*}{$\zero_{q \times 2}$}} &  \multicolumn{1}{c}{\multirow{3}{*}{$\cdots$}}&  \multicolumn{2}{c|}{\multirow{3}{*}{$\zero_{q \times 2}$}} & \multicolumn{3}{c}{\multirow{3}{*}{$\one_{q \times q}$}}\\
            $\vdots$& &&&& &&&&& &&\\
            $f_q$& &&&& &&&&& &&\\
            \bottomrule
        \end{tabular}
        \caption{The table for even $p$.}
    \end{subfigure}
    \newline
    \
    \newline
    \begin{subfigure}{\textwidth}
        \centering
        \small
        \begin{tabular}{ c | cccccc | cc}
            \toprule
            & $v_1$ & \multicolumn{4}{c}{$\cdots$} &$v_{p}$ & $w_1$ & $w_{2}$	\\	
            \midrule
            $f_1$& \multicolumn{2}{c}{\multirow{2}{*}{$\one_{2 \times 2}$}} & \multicolumn{1}{c}{\multirow{2}{*}{$\cdots$}} &  \multicolumn{2}{c}{\multirow{2}{*}{$\one_{2 \times 2}$}} & $1$ & \multicolumn{2}{c}{\multirow{2}{*}{$\one_{2 \times 2}$}}\\
            $f_2$& &&&& &$-1$&&\\
            \bottomrule
        \end{tabular}
        \caption{The table for odd $p$ and $q=2$. }
    \end{subfigure}
    \newline
    \
    \newline
    \begin{subfigure}{\textwidth}
        \centering
        \small
        \begin{tabular}{ c | cccc | ccccc | ccc}
            \toprule
            & $v_1$ & \multicolumn{2}{c}{$\cdots$} &$v_{2q+1}$ & $v_{2q+2}$ & \multicolumn{3}{c}{$\cdots$} &$v_{p}$ &$w_1$ & $\cdots$ & $w_{q}$	\\	
            \midrule
            $f_1$& \multicolumn{2}{c}{\multirow{3}{*}{$\one_{q \times q}$}} & \multicolumn{2}{c|}{\multirow{3}{*}{$\onep_{q \times (q+1)}$}} &  \multicolumn{2}{c}{\multirow{3}{*}{$\zero_{q \times 2}$}} &  \multicolumn{1}{c}{\multirow{3}{*}{$\cdots$}}&  \multicolumn{2}{c|}{\multirow{3}{*}{$\zero_{q \times 2}$}} & \multicolumn{3}{c}{\multirow{3}{*}{$\one_{q \times q}$}}\\
            $\vdots$& &&&& &&&&& &&\\
            $f_q$& &&&& &&&&& &&\\
            \bottomrule
        \end{tabular}
        \caption{The table for odd $p$ and $q \neq 2$.}
    \end{subfigure}
    \caption{An STRD family $\{f_1, \ldots, f_{q}\}$ on $K_{p,q \times 1}$  for pairs $(p,q)$ considered in Case 6.}
    \label{tab: STRD family on K p q when p geq 2q}
    \end{figure}

\begin{figure}[h]
    \centering
    \small
    \begin{tabular}{ c | ccc | cccccc | ccc | ccc}
        \toprule	
            & $v_1$ & $\cdots$ & $v_p$ & $v_{p+1}$ & \multicolumn{4}{c}{$\cdots$} & $v_{p+12}$ & $w_1$ & $\cdots$ & $w_q$ & $w_{q+1}$ & \multicolumn{1}{c}{$\cdots$} & $w_{q+6}$\\
            \midrule
            $f_1$ & \multicolumn{3}{c|}{\multirow{3}{*}{$\mathbb{P}(p,q)$}} & \multicolumn{3}{c}{\multirow{3}{*}{$\zero_{d \times 6}$}}  & \multicolumn{3}{c|}{\multirow{3}{*}{$\zero_{d \times 6}$}} & \multicolumn{3}{c|}{\multirow{3}{*}{$\mathbb{Q}(p,q)$}}& \multicolumn{3}{c}{\multirow{3}{*}{$\zero_{d \times 6}$}} \\
            $\vdots$ &&& &&&&&& &&& &&&\\
            $f_d$ &&& &&&&&& &&& &&&\\
            \midrule
            $f_{d+1}$ & \multicolumn{3}{c|}{\multirow{3}{*}{$\zero_{p \times 6}^T$}} & \multicolumn{3}{c}{\multirow{3}{*}{$\one_{6 \times 6}$}} & \multicolumn{3}{c|}{\multirow{3}{*}{$\one_{6 \times 6}$}} & \multicolumn{3}{c|}{\multirow{3}{*}{$\zero_{q \times 6}^T$}} & \multicolumn{3}{c}{\multirow{3}{*}{$\one_{6 \times 6}$}} \\
            $\vdots$ &&& &&&&&& &&& &&&\\
            $f_{d+6}$ &&& &&&&&& &&& &&&\\
        \bottomrule
    \end{tabular}
    \caption{An STRD family $\{f_1, \ldots, f_{d+6}\}$ on $K_{p+12,(q+6) \times 1}$ for pairs $(p,q)$ considered in Case 7 with $d = \lfloor (p+q)/3\rfloor$.
    The $p+12$ left columns form the matrix $\mathbb{P}(p+12,q+6)$ and the $q+6$ right columns form the matrix $\mathbb{Q}(p+12,q+6)$.}
    \label{tab: STRD family on K p plus 12 q plus 6}
\end{figure}

\begin{table}[h]
    \centering
    \small
    \begin{tabular}{ cc | c|c|c|c | c|c|c|c|c | c|c|c|c|c | c|c }
        \toprule	
        &&  \multicolumn{16}{c}{$q$}\\
            && $2$&$3$&$4$&$5$& $6$&$7$&$8$&$9$&$10$& $11$&$12$&$13$&$14$&$15$&$16$&$17$\\
            \cmidrule{1-18}
            \cmidrule{2-18}
            \multicolumn{1}{c}{\multirow[c]{38}{*}{$p$}}&$2$& $\heartsuit$&$\heartsuit$&$\clubsuit$&$\heartsuit$ &$\clubsuit$&$\heartsuit$&$\heartsuit$&$\heartsuit$&$\diamondsuit$&$\clubsuit$&$\diamondsuit$&$\clubsuit$&$\clubsuit$&$\clubsuit$&$\diamondsuit$&$\diamondsuit$   \\
            \cmidrule{2-18}
            &$3$&$\heartsuit$&$\heartsuit$&$\heartsuit$&$\clubsuit$&$\heartsuit$ &$\heartsuit$&$\heartsuit$&$\clubsuit$&$\clubsuit$&$\diamondsuit$&$\clubsuit$&$\clubsuit$&$\clubsuit$&$\diamondsuit$&$\diamondsuit$&$\diamondsuit$  \\
            \cmidrule{2-18}
            &$4$&$\clubsuit$&$\heartsuit$&$\clubsuit$&$\heartsuit$&$\heartsuit$ &$\heartsuit$&$\diamondsuit$&$\clubsuit$&$\diamondsuit$&$\clubsuit$ &$\clubsuit$&$\clubsuit$&$\diamondsuit$&$\diamondsuit$&$\diamondsuit$&$\diamondsuit$ \\
            \cmidrule{2-18}
            &$5$& $\heartsuit$&$\heartsuit$&$\heartsuit$&$\heartsuit$&$\heartsuit$ &$\clubsuit$&$\clubsuit$ &$\clubsuit$&$\clubsuit$& $\clubsuit$&$\clubsuit$&$\diamondsuit$&$\diamondsuit$&$\diamondsuit$&$\diamondsuit$&$\diamondsuit$    \\
            \cmidrule{2-18}

            &$6$& $\clubsuit$&$\heartsuit$&$\heartsuit$&$\heartsuit$& $\clubsuit$&$\clubsuit$&$\diamondsuit$&$\clubsuit$&$\clubsuit$& $\clubsuit$&$\diamondsuit$&$\diamondsuit$&$\diamondsuit$&$\diamondsuit$&$\diamondsuit$&$\diamondsuit$ \\
            \cmidrule{2-18}
            &$7$& &$\heartsuit$&$\heartsuit$&$\heartsuit$& $\clubsuit$&$\clubsuit$&$\clubsuit$&$\clubsuit$&$\clubsuit$& $\clubsuit$&$\diamondsuit$&$\diamondsuit$&$\diamondsuit$&$\diamondsuit$&$\diamondsuit$&$\diamondsuit$ \\
            \cmidrule{2-18}
            &$8$& &$\heartsuit$&$\clubsuit$&$\heartsuit$& $\clubsuit$&$\clubsuit$&$\clubsuit$&$\clubsuit$&$\diamondsuit$&$\clubsuit$&$\diamondsuit$&$\diamondsuit$&$\diamondsuit$&$\diamondsuit$&$\diamondsuit$&$\diamondsuit$  \\
            \cmidrule{2-18}
            &$9$& &&$\clubsuit$&$\heartsuit$& $\clubsuit$&$\clubsuit$&$\clubsuit$&$\clubsuit$&$\diamondsuit$&$\clubsuit$&$\diamondsuit$&$\diamondsuit$&$\diamondsuit$&$\diamondsuit$&$\diamondsuit$&$\diamondsuit$  \\
            \cmidrule{2-18}
            &$10$& &&$\clubsuit$&$\clubsuit$& $\clubsuit$&$\clubsuit$&$\clubsuit$&$\clubsuit$&$\diamondsuit$&$\diamondsuit$&$\diamondsuit$&$\diamondsuit$&$\diamondsuit$&$\diamondsuit$&$\diamondsuit$&$\diamondsuit$  \\

            \cmidrule{2-18}

            &$11$& &&&$\clubsuit$&$\clubsuit$ &$\heartsuit$&$\clubsuit$&$\clubsuit$&$\clubsuit$&$\diamondsuit$&$\diamondsuit$&$\clubsuit$&$\diamondsuit$&$\diamondsuit$&$\diamondsuit$&$\diamondsuit$  \\
            \cmidrule{2-18}
            &$12$& &&&$\clubsuit$&$\clubsuit$ &$\heartsuit$&$\clubsuit$&$\clubsuit$&$\clubsuit$&$\diamondsuit$&$\diamondsuit$&$\clubsuit$&$\diamondsuit$&$\diamondsuit$&$\diamondsuit$&$\diamondsuit$  \\
            \cmidrule{2-18}
            &$13$& &&&&$\clubsuit$ &$\heartsuit$&$\clubsuit$&$\clubsuit$&$\clubsuit$& $\clubsuit$&$\diamondsuit$&$\clubsuit$&$\diamondsuit$&$\diamondsuit$&$\diamondsuit$&$\diamondsuit$ \\
            \cmidrule{2-18}
            &$14$& &&&&$\clubsuit$ &$\clubsuit$&$\clubsuit$&$\clubsuit$&$\diamondsuit$& $\clubsuit$&$\diamondsuit$&$\diamondsuit$&$\diamondsuit$&$\diamondsuit$&$\diamondsuit$&$\diamondsuit$   \\
            \cmidrule{2-18}
            &$15$& &&&& &$\clubsuit$&$\clubsuit$&$\clubsuit$&$\clubsuit$&$\diamondsuit$ &$\clubsuit$&$\diamondsuit$&$\diamondsuit$&$\diamondsuit$&$\diamondsuit$&$\diamondsuit$    \\

            \cmidrule{2-18}

            &$16$& &&&& &$\clubsuit$&$\diamondsuit$ &$\clubsuit$ &$\diamondsuit$ &$\clubsuit$ &$\clubsuit$ &$\diamondsuit$&$\diamondsuit$&$\diamondsuit$&$\diamondsuit$&$\diamondsuit$   \\
            \cmidrule{2-18}
            &$17$& &&&& &&$\clubsuit$&$\clubsuit$&$\clubsuit$& $\clubsuit$&$\clubsuit$&$\diamondsuit$&$\diamondsuit$&$\diamondsuit$&$\diamondsuit$&$\diamondsuit$  \\
            \cmidrule{2-18}
            &$18$& &&&& &&$\diamondsuit$&$\clubsuit$&$\clubsuit$& $\clubsuit$&$\diamondsuit$&$\diamondsuit$&$\diamondsuit$&$\diamondsuit$&$\diamondsuit$&$\diamondsuit$  \\
            \cmidrule{2-18}
            &$19$& &&&& &&&$\clubsuit$&$\clubsuit$& $\clubsuit$&$\diamondsuit$&$\diamondsuit$&$\diamondsuit$&$\diamondsuit$&$\diamondsuit$&$\diamondsuit$\\
            \cmidrule{2-18}
            &$20$& &&&& &&&$\clubsuit$&$\diamondsuit$& $\clubsuit$&$\diamondsuit$&$\diamondsuit$&$\diamondsuit$&$\diamondsuit$&$\diamondsuit$&$\diamondsuit$ \\

            \cmidrule{2-18}

            &$21$& &&&& &&&&$\diamondsuit$& $\clubsuit$&$\diamondsuit$&$\diamondsuit$&$\diamondsuit$&$\diamondsuit$&$\diamondsuit$&$\diamondsuit$  \\
            \cmidrule{2-18}
            &$22$& &&&& &&&&$\diamondsuit$&$\diamondsuit$&$\diamondsuit$&$\diamondsuit$&$\diamondsuit$&$\diamondsuit$&$\diamondsuit$&$\diamondsuit$ \\

            \cmidrule{2-18}
            &$23$& &&&& &&&&&$\diamondsuit$&$\diamondsuit$&$\clubsuit$&$\diamondsuit$&$\diamondsuit$&$\diamondsuit$&$\diamondsuit$   \\
            \cmidrule{2-18}
            &$24$& &&&& &&&&&$\diamondsuit$&$\diamondsuit$&$\clubsuit$ &$\diamondsuit$&$\diamondsuit$&$\diamondsuit$&$\diamondsuit$  \\
            \cmidrule{2-18}
            &$25$& &&&& &&&&&&$\diamondsuit$ &$\clubsuit$&$\diamondsuit$&$\diamondsuit$&$\diamondsuit$&$\diamondsuit$ \\

        \bottomrule
    \end{tabular}
    \caption{
    Overview of the inductive base in Case 7. 
    $\heartsuit$ indicates pairs $(p,q)$ with $p \leq 2q+2$ where one of Cases 1 to 5 applies.
    $\clubsuit$ indicates pairs $(p,q)$ with $p \leq 2q+2$ contained in the inductive base.
    $\diamondsuit$ indicates pairs $(p,q)$ with $p \leq 2q+2$ where $d_{stR}(K_{p,q \times 1}) = \lfloor (p+q)/3 \rfloor$ follows via the inductive step.
    Entries $(p,q)$ with $p \geq 2q+3$ are left blank.
    }
    \label{tab: inductive base d str K p q}
\end{table}

\begin{figure}[]
    \centering
    \small
    \begin{subfigure}{\textwidth}
        \centering
        \begin{tabular}{ c | ccc | ccc | cccccc}
            \toprule
            & $v_1$ & \multicolumn{1}{c}{$\cdots$} &$v_{p}$ & $w_{1}$ & \multicolumn{1}{c}{$\cdots$} &$w_{q}$ &$w_{q+1}$ & \multicolumn{4}{c}{$\cdots$} & $w_{q+6}$	\\	
            \midrule
            $f_1$& \multicolumn{3}{c|}{\multirow{3}{*}{$\mathbb{P}(p,q)$}} & \multicolumn{3}{c|}{\multirow{3}{*}{$\mathbb{Q}(p,q)$}} & \multicolumn{6}{c}{\multirow{3}{*}{$\zero_{d \times 6}$}} \\
            $\vdots$& &&& &&& &&&&&\\
            $f_d$& &&& &&& &&&&&\\
            \midrule
            $f_{d+1}$& \multicolumn{3}{c|}{\multirow{2}{*}{$\zero_{p \times 2}^T$}} &  \multicolumn{3}{c|}{\multirow{2}{*}{$\zero_{q \times 2}^T$}} & \multicolumn{2}{c}{\multirow{2}{*}{$\one_{2 \times 2}$}} & \multicolumn{2}{c}{\multirow{2}{*}{$\one_{2 \times 2}$}} & \multicolumn{2}{c}{\multirow{2}{*}{$\one_{2 \times 2}$}}\\
            $f_{d+2}$& &&& &&& &&&&&\\
            \bottomrule
        \end{tabular}
        \caption{The table for even $p$ and even $q$.}
    \end{subfigure}
    \newline
    \
    \newline
    \begin{subfigure}{\textwidth}
        \centering
        \begin{tabular}{ c | ccc | cccc | cccccc}
            \toprule
            & $v_1$ & \multicolumn{1}{c}{$\cdots$} &$v_{p}$ & $w_{1}$ & \multicolumn{2}{c}{$\cdots$} &$w_{q}$ &$w_{q+1}$ & \multicolumn{4}{c}{$\cdots$} & $w_{q+6}$	\\	
            \midrule
            $f_1$& \multicolumn{3}{c|}{\multirow{3}{*}{$\mathbb{P}(p,q)$}} & \multicolumn{4}{c|}{\multirow{3}{*}{$\mathbb{Q}(p,q)$}} & \multicolumn{6}{c}{\multirow{3}{*}{$\zero_{d \times 6}$}} \\
            $\vdots$& &&& &&&& &&&&&\\
            $f_d$& &&& &&&& &&&&&\\
            \midrule
            $f_{d+1}$& \multicolumn{2}{c}{\multirow{2}{*}{$\zero_{(p-1) \times 2}^T$}} & $-1$ & \multicolumn{2}{c}{\multirow{2}{*}{$\zero_{(q-2) \times 2}^T$}} & $-1$ & $-1$ & \multicolumn{2}{c}{\multirow{2}{*}{$\one_{2 \times 2}$}} & \multicolumn{2}{c}{\multirow{2}{*}{$\one_{2 \times 2}$}} & $2$ & $2$ \\
            $f_{d+2}$& &&$1$& &&$1$&$1$& &&&&$-1$&$-1$\\
            \bottomrule
        \end{tabular}
        \caption{The table for odd $p$ and even $q$.}
    \end{subfigure}
    \newline
    \
    \newline
    \begin{subfigure}{\textwidth}
        \centering
        \begin{tabular}{ c | cccc | ccc | cccccc}
            \toprule
            & $v_1$ & \multicolumn{2}{c}{$\cdots$} &$v_{p}$ & $w_{1}$ & \multicolumn{1}{c}{$\cdots$} &$w_{q}$ &$w_{q+1}$ & \multicolumn{4}{c}{$\cdots$} & $w_{q+6}$	\\	
            \midrule
            $f_1$& \multicolumn{4}{c|}{\multirow{3}{*}{$\mathbb{P}(p,q)$}} & \multicolumn{3}{c|}{\multirow{3}{*}{$\mathbb{Q}(p,q)$}} & \multicolumn{6}{c}{\multirow{3}{*}{$\zero_{d \times 6}$}} \\
            $\vdots$& &&& &&&& &&&&&\\
            $f_d$& &&& &&&& &&&&&\\
            \midrule
            $f_{d+1}$& \multicolumn{2}{c}{\multirow{2}{*}{$\zero_{(p-2) \times 2}^T$}} & $-1$ & $-1$ &\multicolumn{2}{c}{\multirow{2}{*}{$\zero_{(q-1) \times 2}^T$}} & $-1$ & \multicolumn{2}{c}{\multirow{2}{*}{$\one_{2 \times 2}$}} & \multicolumn{2}{c}{\multirow{2}{*}{$\one_{2 \times 2}$}} & $2$ & $2$ \\
            $f_{d+2}$& &&$1$&$1$ &&&$1$& &&&&$-1$&$-1$\\
            \bottomrule
        \end{tabular}
        \caption{The table for even $p$ and odd $q$.}
    \end{subfigure}
    \newline
    \
    \newline
    \begin{subfigure}{\textwidth}
        \centering
        \begin{tabular}{ c | ccc | ccc | cccccc}
            \toprule
            & $v_1$ & \multicolumn{1}{c}{$\cdots$} &$v_{p}$ & $w_{1}$ & \multicolumn{1}{c}{$\cdots$} &$w_{q}$ &$w_{q+1}$ & \multicolumn{4}{c}{$\cdots$} & $w_{q+6}$	\\	
            \midrule
            $f_1$& \multicolumn{3}{c|}{\multirow{3}{*}{$\mathbb{P}(p,q)$}} & \multicolumn{3}{c|}{\multirow{3}{*}{$\mathbb{Q}(p,q)$}} & \multicolumn{6}{c}{\multirow{3}{*}{$\zero_{d \times 6}$}} \\
            $\vdots$& &&& &&& &&&&&\\
            $f_d$& &&& &&& &&&&&\\
            \midrule
            $f_{d+1}$& \multicolumn{2}{c}{\multirow{2}{*}{$\zero_{(p-1) \times 2}^T$}} & $-1$ &\multicolumn{2}{c}{\multirow{2}{*}{$\zero_{(q-1) \times 2}^T$}} & $1$ & \multicolumn{2}{c}{\multirow{2}{*}{$\one_{2 \times 2}$}} & \multicolumn{2}{c}{\multirow{2}{*}{$\one_{2 \times 2}$}} & \multicolumn{2}{c}{\multirow{2}{*}{$\one_{2 \times 2}$}} \\
            $f_{d+2}$& &&$1$& &&$-1$& &&&&&\\
            \bottomrule
        \end{tabular}
        \caption{The table for odd $p$ and odd $q$.}
    \end{subfigure}
    \caption{An STRD family $\{f_1, \ldots, f_{d+2}\}$ on $K_{p,(q+6) \times 1}$ for pairs $(p,q)$ considered in Case 7 with $d = \lfloor (p+q)/3\rfloor$.
    The $p$ left columns form the matrix $\mathbb{P}(p,q+6)$ and the $q+6$ right columns form the matrix $\mathbb{Q}(p,q+6)$.}
    \label{tab: STRD family on K p q plus 6}
\end{figure}

\clearpage
\section{Open Problems}

We end with some open problems. 
In view of \cref{thm:NPcompleteCubicPlanar,thm:NPcompleteCubicBipartite} it appears natural to ask whether \textsc{Open Packing} remains hard for the intersection of cubic planar and cubic bipartite graphs.

\begin{problem}
    Is \textsc{Open Packing} \NP-complete even for cubic planar bipartite graphs?
\end{problem}

With \cref{prop: str domination number complete bipartite graph,thm: str domination number complete multipartite graphs} we now know the value $\gamma_{stR}(G)$ for every complete multipartite graph $G$.
On the other hand, \cref{prop: str domatic number complete graph,prop: str domatic number complete bipartite graph,thm: d str K 3 3 3,thm: d str K 2 2 q times 1,thm: d str K p q times 1} determine $d_{stR}(G)$ only if $G$ is a complete bipartite graph or a complete $r$-partite graph with $r \geq 3$ and $\gamma_{stR}(G) = 3$.
It would be desirable to settle the remaining cases.

\begin{problem}
    Determine $d_{stR}(G)$ for all complete $r$-partite graphs $G$ with $r \geq 3$ and $\gamma_{stR}(G) = 2$.
\end{problem}

\section*{Acknowledgements}

Research of the third author was supported by the RWTH Graduate Support.

\printbibliography
\end{document}